\documentclass[reqno, 12pt]{amsart}
\usepackage{mathrsfs}
\usepackage{amscd}
\usepackage{amsmath}
\usepackage{latexsym}
\usepackage{amsfonts}
\usepackage{amssymb}
\usepackage{amsthm}
\usepackage{graphicx}

\usepackage{hyperref}

\newtheorem{theorem}{Theorem}[section]
\newtheorem{proposition}[theorem]{Proposition}
\newtheorem{corollary}[theorem]{Corollary}
\newtheorem{lemma}[theorem]{Lemma}
\newtheorem{definition}[theorem]{Definition}

\newtheorem{remark}[theorem]{Remark}

\numberwithin{equation}{section}

\title[Asymptotic expansion at zero-resonance ]
{Time Asymptotic expansions of solution for fourth-order Schr\"odinger equation with zero resonance or eigenvalue}

\author{Hongliang Feng, \ Zhao Wu \ and Xiaohua Yao}

\address {Hongliang Feng, School of Mathematics, Sun Yat-sen University, Guangzhou, 510275, P. R. China}
\email{fenghongliang@aliyun.com}
\address{Zhao Wu, School of Mathematics and Statistics, Central China
Normal University, Wuhan, 430079, P. R. China}
\email{wuzhao218@yahoo.com}
\address{Xiaohua Yao, Department of Mathematics and  Hubei Province Key Laboratory of Mathematical Physics,
 Central China Normal University, Wuhan, 430079, P.R. China}
\email{yaoxiaohua@mail.ccnu.edu.cn}
\date{\today}
\keywords{Fourth-order Schr\"odinger operator, lower energy asymptotic expansion, classification of resonance subspaces, Kato-Jensen decay estimates}

\begin{document}

\begin{abstract}\baselineskip=11pt
 In this paper, we first deduce  the  asymptotic expansions of resolvent of $H=(-\Delta)^2+V$ with the  presence of  resonance or eigenvalue at the degenerate zero threshold  for $d\geq5$. In particular, we identify these resonance spaces for full kinds of zero resonances. As a consequence, we then establish the {\it time asymptotic expansions} and {\it Kato-Jensen estimates} for the solution of fourth-order Schr\"odinger equation under the presence of zero resonance or eigenvalue.
\end{abstract}

\maketitle
\baselineskip=18pt
\section{Introduction}
\subsection{Backgrounds}
In this paper, we consider the fourth-order Schr\"odinger operator
$$H=(-\Delta)^2+V,\,\,  H_{0}=(-\Delta)^2$$
where $V(x)$ is a real valued function $\mathbf{R}^d$  satisfying $|V(x)|\lesssim (1+|x|)^{-\beta}$ for some $\beta>0$ depending on $d$.
It was well-known that the biharmonic operator has appeared in many interesting contexts£¬ which include  the (nonlinear) fourth-order Schr\"odinger equation introduced by Karpman \cite{Karpman1, Karpman2}, the nonlinear beam equation or fourth-order wave equation involving  in the study of plate and beams \cite{Love}, in  the study of interaction of water waves \cite{Bretherton}, and in the study of the motion of a suspension bridge, see e.g.  MacKenna and Walter \cite{MW1}.  Moreover, they naturally appear in many other areas of mathematics too, including conformal
geometry (Paneitz operator, Q-curvature \cite{Chang1, Chang2}), free boundary problems\cite{Adams}, and have enjoyed increasing attention in the recent years (see e.g. Mayboroda and  Maz'ya \cite{May-Maz, May-Maz2} and therein references).

In the sequel, we aim to establish the time asymptotic expansions of  $e^{itH}$ with the presence of zero resonance or eigenvalue as  $t$ goes to infinity.
Recall that for Schr\"odinger operator $-\Delta+V$,  Kato and Jensen in the pioneering work \cite{JK} first established the time asymptotic expansion
\begin{equation}\label{asmpexp}
e^{it(-\Delta+V)}=\sum_{j=0}^{N}e^{it\lambda_{j}}P_{j}+t^{-1/2}C_{-1}+t^{-3/2}C_{0}+\cdots
\end{equation}
as $t\rightarrow\infty$  in the weighted Lebesgue spaces for dimension $d=3$. Here the $\lambda_{j}$ are nonnegative eigenvalues of $-\Delta+V$ with associated eigen-projection $P_{j}$.   The spectral property at zero threshold of $-\Delta+V$ affects the leading term of the asymptotic expansion.  In \cite{JK}, they pointed that $C_{-1}=0$  if $0$ is a regular point. However, the operator $C_{-1}$ does not vanish  if $0$ is purely a resonance of $-\Delta+V$.
In \cite{JK},  they   further discussed the following four cases: {\it zero is a regular point; zero is purely a resonance; zero is purely an eigenvalue; zero is both resonance and eigenvalue}. Zero is a  regular point of $-\Delta+V$ means zero is neither an eigenvalue nor a resonance.
The main idea behind asymptotic expansion \eqref{asmpexp} is as following:
\begin{equation*}
\left.
\begin{array}{cc}
\text{Lower energy expansion} \stackrel{LAP}{\Longrightarrow} dE(\lambda)(\lambda\rightarrow0)\\
\text{Higher energy decay estimate} \stackrel{LAP}{\Longrightarrow} dE(\lambda)(\lambda\rightarrow+\infty)
\end{array}
\right\}
{\Longrightarrow} \text{time asymptotic expansion}.
\end{equation*}
The lower energy expansion means the asymptotic expansion of perturbed resolvent $R(-\Delta+V; z)$ for $z$ near zero in the weighted Lebesgue spaces. Higher energy decay estimate means the decay estimate of $R(-\Delta+V; z)$ for $z\rightarrow \infty$ in suitable weighted Lebesgue spaces. LAP is short for {\it limiting absorption principle} which is about the existence and continuity of $R(-\Delta+V; z)$ when $z$ close to the positive real axis, e.g. see \cite{Agmon, Ben-Devi}.

The time asymptotic expansion \eqref{asmpexp} implies the following time decay estimate of $-\Delta+V$:
\begin{equation*}
\big\|(1+|x|)^{-s}e^{it(-\Delta+V)}P_{ac}(1+|x|)^{-s'}f\big\|_{L^2(\mathbf{R}^3)}\lesssim (1+|t|)^{-3/2}\|f\|_{L^2(\mathbf{R}^3)}, \,\, t\in \mathbf{R},
\end{equation*}
with zero is a regular point of $-\Delta+V$ and suitable $s, s'>0$.
However, with the presence of zero resonance or eigenvalue, Kato and Jensen shows that the time decay rate of Kato-Jensen dispersive estimate is slower than the regular case. The time decay is $(1+|t|)^{-3/2}$ in the regular case, while if zero is purely a resonance or both eigenvalue and resonance, then the decay rate is also $(1+|t|)^{-1/2}$ by the expansion results in \cite{JK}.  For other dimensional cases,  one can see Jensen's works \cite{J1, J} for $d=4$ and $d\geq5$ respectively, and \cite{JN} for  $d=1, 2$.  For $d\geq5$, Jensen in \cite{J} pointed out that $-\Delta+V$ has no zero resonance. Later, Murata in \cite{MM}  generalized Kato and Jensen's works \cite{JK, J, J1} to the operator $P(D)+V$, where $P(D)$ is an $m$-order elliptic differential operator with real constant coefficients and satisfies the following {\it non-degenerate conditions}:
$$\det\big[\partial_{i}\partial_{j}P(\xi)\big]\Big|_{\xi_0}\neq0 \quad \text{for} \quad \big(\nabla P\big)(\xi_0)=0.$$
However, note that if $P(D)=(-\Delta)^m$ with $m\ge 2$, they do not satisfy the above non-degenerate conditions at zero but the case $m=1$. Hence the results in \cite{MM} did not cover these simple poly-harmonic operators.

For the fourth-order Schr\"odinger operator $H=(-\Delta)^2+V$, Soffer, the first and third author  in \cite{FSY} first gave the asymptotic expansions of the resolvent of fourth-order Schr\"odinger operator $(-\Delta)^2+V$ around zero threshold assume that  zero is a regular point of $H$ for $d\geq5$ and $d=3$. Their results show that, in the regular case, the time decay rate of Kato-Jensen decay estimate is $(1+|t|)^{-d/4}$ with $d\geq5$. In the case $d=4$, Green and Toprak in \cite{Green-Topark} derived the asymptotic expansion of the resolvent $R_{V}(z)$ near zero with the presence of resonance or eigenvalue. Indeed, with the help of higher energy decay estimates established in \cite{FSY}, the time decay rate of Kato-Jensen decay estimate is $(1+|t|)^{-1}$ if zero is a regular point of $H=(-\Delta)^2+V$. While, the time decay rate is $(\ln |t|)^{-1}$ as $t\rightarrow\infty $ if zero is purely an eigenvalue of $H$.

Motivated  by  these interesting works \cite{FSY, Green-Topark, JK} above,  this paper will be devoted to establish the {\it time asymptotic expansion} of $e^{itH}P_{ac}$ and {\it Kato-Jensen estimates } for the fourth-order Schr\"odinger operator $H=(-\Delta)^2+V$  with {\it the presence of zero resonance or eigenvalue } for $d\geq5$  in suitable weighted Lebesgue space. Moreover, we can  give {\it a full classification of the different resonances at zero threshold, which reflect exactly  different effects on decay rates of time}. For the end,  we first deduce  the  asymptotic expansions of resolvent of $H=(-\Delta)^2+V$ with the  presence of  resonance or eigenvalue at the degenerate zero threshold by making  use of the following symmetric resolvent identity
\begin{equation}\label{Born}
R_{V}(z)=R_{0}(z)-R_{0}(z)v\big(U+vR_{0}(z)v\big)^{-1}vR_{0}(z)
\end{equation}
 Here, $v(x)=|V(x)|^{1/2}$ and $U={\rm sign}\big(V(x)\big)$. By the formula \eqref{Born} and free resolvent $R_0(z)$ of $\Delta^2$, it can be reduced to get the asymptotic expansion of $\big(U+vR_{0}(z)v\big)^{-1}$.
If the first term $T_{0}$ of the asymptotic series of $U+vR_{0}(z)v$ is invertible on $L^{2}$, then we can expand $\big(U+vR_{0}(z)v\big)^{-1}$ into Neumann series.  If not, then $T_1=T_{0}+S_{1}$ is invertible on $L^{2}$ where $S_{1}$ is the projection onto $\ker (T_{0})$.   If $T_{1}$ is invertible, then we done. If not, we need to go on the iterative process. We prove that the iterative process stops for finite steps. This processes are related to the  the classification of zero resonance. In fact, due to the degenerate of $H_{0}=(-\Delta)^2$ at zero, the classification of zero resonance is more complex than Schr\"odinger operator $-\Delta+V$.  For $d=7, 8$, there is only one kind of zero resonance for $H$. while for $d=5, 6$, there are two kinds of zero resonances for $H$.     For $d\geq9$, there is no zero resonance for $H=(-\Delta)^2+V$.

Finally, we remark that,  due to the order of $H=(-\Delta)^2+V$ is higher than Schr\"odinger operator $-\Delta+V$, the iterative process does not works well at the moment for $H=(-\Delta)^2+V$ with $d=1, 2, 3$.  The obstacle is  that the lower order terms of the asymptotic series of  $U+vR_{0}(z)v$ interact each other,  which leads difficulty to analyze $\ker(T_{j})$ when $d=1, 2, 3$. However, such obstacle does not exist for Schr\"odinger operator when $d\ge 1$.


\subsection{Main results}


Throughout this paper, $P_{ac}$ denotes the projection onto the absolutely continuous spectrum space of $H$.  For $a\in\mathbf{R}$, we write $a\pm$ to mean $a\pm\epsilon$ for any tiny $\epsilon>0$.  For $s, s'\in\mathbf{R}$, $B(s, s')$ denote the bounded family of $L^{2}_{s}(\mathbf{R}^n)$ to $L^{2}_{s'}(\mathbf{R}^n)$ where $L^{2}_{s}(\mathbf{R}^n)$ is the weighted Lebesgue space:
$$L^{2}_{s}(\mathbf{R}^d)=\Big\{f\in L^{2}(\mathbf{R}^d): (1+|\cdot|)^{s}f\in L^{2}(\mathbf{R}^d)\Big\}.$$

For $\sigma\in\mathbf{R}$,  let $W_{\sigma}(\mathbf{R}^d)$ denotes the intersection  space
$$ W_{\sigma}(\mathbf{R}^d)= \cap_{s>\sigma} L^{2}_{-s}(\mathbf{R}^d).$$
Clearly, $W_{\sigma_1}(\mathbf{R}^d)\supset W_{\sigma_2}(\mathbf{R}^d)$ if $\sigma_1>\sigma_2$ since $L^{2}_{s}(\mathbf{R}^d)$ is monotonically decreasing about $s\in\mathbf{R}$.  Especially, $W_{0}(\mathbf{R}^d)\supset L^{2}(\mathbf{R}^d)$. Note that $d\ge 9$, $(-\Delta)^2+V$ has no any zero resonance.  So we only give the definition of zero resonance for $d\le 8$.

\begin{definition}\label{resonance'}
For $H=(-\Delta)^2+V$ with $|V(x)|\lesssim (1+|x|)^{-\beta}$ with some $\beta>0$.

(1) For $d=5$, if there exists
 a $\psi\in W_{3/2}(\mathbf{R}^5) \setminus W_{1/2}(\mathbf{R}^5)$
such that $H\psi=0$ in distributional sense, then we say zero is of the first kind resonance of $H$.  If there exists a
 $\psi\in W_{1/2}(\mathbf{R}^5)\setminus L^{2}(\mathbf{R}^5)$
 such that $H\psi=0$, then we say zero is of the second kind resonance of $H$.

(2) For $d=6$, if there exists a
$\psi\in W_{1}(\mathbf{R}^6) \setminus W_{0}(\mathbf{R}^6)$ such that $H\psi=0$, then we say zero is of the first kind resonance of $H$.  If there exists
$\psi\in W_{0}(\mathbf{R}^6)\setminus L^{2}(\mathbf{R}^6)$
such that $H\psi=0$, then we say zero is of the second kind resonance of $H$.

(3) For $d=7$, if there exists
$\psi\in W_{1/2}(\mathbf{R}^7)\setminus L^{2}(\mathbf{R}^7)$
such that $H\psi=0$, then we say zero is a resonance of $H$.

(4) For $d=8$, if there exists
$\psi\in W_{0}(\mathbf{R}^6)\setminus L^{2}(\mathbf{R}^8) $ such that $H\psi=0$, then we say zero is a resonance of $H$.
\end{definition}

 For the definition above, several remarks can be given as follows:

\begin{remark}
Note that $H\psi=\big((-\Delta)^2+V\big)\psi=0$ in the distributional sense if and only if $$ \big(1+R_{0}(0)V\big)\psi=0\  {\rm or }\ \big (1+VR_{0}(0)\big)\psi=0$$ where $R_{0}(0)=(-\Delta)^{-2}$ for $d\ge 5$. Set
$$\mathfrak{M}_s=\Big\{\psi\in L^{2}_{-s}: \big(1+R_{0}(0)V\big)\psi=0\Big\};\,\, \mathfrak{N}_s=\Big\{\psi\in L^{2}_{s}: \big (1+VR_{0}(0)\big )\psi=0\Big\}.$$
If $V(x)$ satisfy $|V(x)|\lesssim(1+|x|)^{-\beta}$ for $\beta>8-d\ge 0 $ and $\psi\in  L^{2}_{-s}(\mathbf{R}^d)$ for $s>4-d/2$, then $V\psi\in  L^{2}_{\beta-s}(\mathbf{R}^d)$.  By the boundness of $R_{0}(0)$ in $B(s, -s')$ (see e.g. \cite[Lemma 2.3]{J} or Lemma \ref{Reisz-potential-boundedness} below), then the identity $\psi=-R_{0}(0)V\psi$ requires that $s\in (4-\frac{d}{2},\,\, \beta+\frac{d}{2}-4)$.
   Note that $\mathfrak{M}_s$ and $\mathfrak{N}_s$ are monotone in $s$ in the opposite direction,  and $\dim(\mathfrak{M})_s=\dim(\mathfrak{N})_s<\infty$ by Fredholm's duality theorem.  Therefore, $\mathfrak{M}_s$ and $\mathfrak{N}_s$ are independent of $s\in (4-\frac{d}{2},\,\, \beta+\frac{d}{2}-4)$. This explains that  the the resonance spaces must belong to the intersection space $\cap_{s>\sigma} L^{2}_{-s}(\mathbf{R}^d)$ with $\sigma=4-d/2$ as $d\le 8$.
\end{remark}

\begin{remark}
	For $d\geq9$,  Riesz potential $R_{0}(0)=(-\Delta)^{-2}$ is bounded from $L^2_{s_1}(\mathbf{R}^d)$ to $L^2_{-s_2}(\mathbf{R}^d)$ with $s_1+s_2\geq 4$ and $s_1, s_2\geq0$. Thus for $\psi\in L^{2}_{-s}(\mathbf{R}^d)$ with some $s\geq0$ satisfying  $\psi=-R_{0}(0)V\psi$ (i.e  $H\psi=0$), $R_{0}(0)=(-\Delta)^{-2}$ can pull $V\psi\in L^{2}_{\beta-s}(\mathbf{R}^d)$ into $L^{2}(\mathbf{R}^d)$ if $\beta-s\geq4$. Note that, $R_{0}(0)=(-\Delta)^{-2}$ is nice in the sence that it is defined on $L^2(\mathbf{R}^d)$ for $d\geq9$. Hence $H=(-\Delta)^2+V$ does not exist resonance  at zero.
\end{remark}
  \begin{remark} Similarly,  for Schr\"odinger operator $-\Delta+V$,  one say zero is resonance if $(-\Delta+V)\psi=0$ for some $\psi\in W_{\frac{1}{2}}(\mathbf{R}^3)\setminus L^2(\mathbf{R}^3)$ ( i.e. $\sigma=1/2$ ) as $d=3$,  and $\psi\in W_{0}(\mathbf{R}^3)\setminus L^2(\mathbf{R}^3)$ ( i.e. $\sigma=0$ ) as $d=4$. In particular, $-\Delta+V$ does not exist resonance function at zero for $d\geq5$ (see e.g. \cite{JK, J, J1}).
\end{remark}

Next, we state the main results about the time asymptotic expansions of $e^{itH}P_{ac}$.
\begin{theorem}\label{eitH}
Let $H=(-\Delta)^2+V$ and  $|V(x)|\lesssim (1+|x|)^{-\beta}$ for  some $\beta>0$ chosen as follows. Assume $H$ has no positive embedded eigenvalue.  Then the following statements holds:

\noindent (1) For $d=5$, let $\beta>2d+1$ and $s>d+1/2$,  then
	
	i) If $0$ is of the first kind resonance of $H$, then we have in $B(s, -s)$
	\begin{equation*}\label{eitH-5-1}
	e^{itH}P_{ac}=|t|^{-3/4}D_{5, 1}^{-1}+|t|^{-1} D_{5, 1}^{0}+|t|^{-5/4} D_{5, 1}^{1}+O(|t|^{-5/4}), \,\, t\rightarrow\infty,
	\end{equation*}
	where $D_{5, 1}^{-j}\in B(s, -s)$ for all $j$.
	
	 ii) If $0$ is of the second resonance or an eigenvalue (with  $k=2, 3$ respectively) of $H$,  then we have in $B(s, -s)$
	\begin{equation*}\label{eitH-5-2}
	e^{itH}P_{ac}=|t|^{-1/4} D_{5, k}^{-3}+|t|^{-1/2} D_{5, k}^{-2}+|t|^{-3/4} D_{5, k}^{-1}+O(|t|^{-3/4}), \,\,  t\rightarrow\infty,
	\end{equation*}
	where $D_{5, k}^{-j}\in B(s, -s)$ for all $j$.
	
\noindent (2) For $d=6$, let $\beta>2d+2$ and $s>d+1$,   then
	
	i) If $0$ is of the first kind resonance, then we have in $B(s, -s)$:
	\begin{equation*}\label{eitH-6-1}
	e^{itH}P_{ac}=|t|^{-1/2} D_{6, 1}^{-2}+|t|^{-1}(\ln|t|)^{-1} D_{6, 1}^{-1}+O(|t|^{-1-}),\,\, t\rightarrow\infty,
	\end{equation*}
	where $D_{6, 1}^{-j}\in B(s, -s)$ for all $j$.
	
	ii) If $0$ is of the second resonance or an eigenvalue (with $k=2, 3$ respectively) of $H$, then we have in $B(s, -s)$:
	\begin{equation*}\label{eitH-6-2}
	e^{itH}P_{ac}=(\ln|t|)^{-1} D_{6, k}^{-4}+(\ln|t|)^{-2} D_{6, k}^{-3}+|t|^{-1/2} D_{6, k}^{-2}+O(|t|^{-1/2}), \,\,  t\rightarrow\infty,
	\end{equation*}
	where $D_{6, k}^{-j}\in B(s, -s)$ for all $j$.	
	
\noindent (3) For $d=7$, let $\beta>d+2$.  If $0$ is a resonance or an eigenvalue (with $k=1, 2$ respectively) of $H$, then for $s>d/2+1$,  we have in $B(s, -s)$:
	\begin{equation*}\label{eitH-7-1}
	e^{itH}P_{ac}=|t|^{-1/4} D_{7, k}^{-3}+|t|^{-1/2}  D_{7, k}^{-2}+|t|^{-3/4} D_{7, k}^{-1}+O(|t|^{-3/4}), \,\,  t\rightarrow\infty,
	\end{equation*}
	where $D_{7, k}^{-j}\in B(s, -s)$ for all $j$.
	
\noindent (4) For $d=8$, let $\beta>d$. If $0$ is a resonance or an eigenvalue (with $k=1, 2$ respectively) of $H$, then for $s>d/2$,  we have in $B(s, -s)$:
	\begin{equation*}\label{eitH-8-1}
	\begin{split}
	e^{itH}P_{ac}=&\frac{1}{\ln|t|} D_{8, k}^{-3}+\frac{1}{(\ln|t|)^2}  D_{8, k}^{-2}+\frac{1}{(\ln|t|)^2}|t|^{-1/2}  D_{8, k}^{-1}+O(|t|^{-1/2-}), \,\, t\rightarrow\infty,
	\end{split}
	\end{equation*}
	where $D_{8, k}^{-j}\in B(s, -s)$ for all $j$.
	
\noindent (5) For $d\geq9$, let $\beta>d$. If $0$ is an eigenvalue of $H$, then for $s>d/2$,  we have in $B(s, -s)$:
	\begin{equation*}\label{eitH-9-1}
	e^{itH}P_{ac}=|t|^{-(d-2)/4}  D_{d, 1}^{-1}+|t|^{-d/4} D_{d, 1}^{0}+O(|t|^{-d/4}), \,\, t\rightarrow\infty,
	\end{equation*}
	where $D_{d, 1}^{-j}\in B(s, -s)$ for all $j$.
\end{theorem}

For the main results above, some remarks are given as follows:
\begin{remark}
 For any $d\geq 1$, one can actually  derive the expansion of $e^{itH}P_{ac}$ with zero is a regular point of $H$ in $B(s, -s)$ with a suitable $s>0$ by the similar processes. For $d\geq 4$, in the regular case, the first term of the expansion deduces that $e^{itH}P_{ac}$ decays like $ |t|^{-d/4}$. However, for $d=1, 2, 3$, the time decay is not $|t|^{-d/4}$ in the regular case.  For example, in 3-dimensional case, the decay rate is $|t|^{-5/4}$, see \cite{FSY}.
\end{remark}

\begin{remark} As for positive embedded eigenvalues, it was well-known that Schr\"odinger operator $-\Delta+V$ does not has any positive embedded eigenvalue for certain decay or $L^p$ integrable potentials, see e.g. Kato \cite{K}, Ionescu and Jerison \cite{IJ}, Koch and Tataru \cite{KoTa} and references therein. However, for fourth order operator $H=(-\Delta)^2+V$,  there exists even some $C_{0}^{\infty}$-potential $V$ such that $1$ is an eigenvalue of $H$.  On the other hand, by applying the Virial argument,  one can also show that $(-\Delta)^2+V$ has no any positive embedded eigenvalues if $V(x)$ is repulsive ( i.e. $ x\cdot \nabla V\le 0$ ), see the follow-up paper \cite{FSWY} for the further studies.
\end{remark}
From the above expansions of $e^{itH}P_{ac}$ in $B(s, -s)$, then we have the following Kato-Jensen type decay estimates for $e^{itH}$.

\begin{corollary}\label{Kato-Jensen}
	For $H=(-\Delta)^2+V$, assume $H$ has no positive embedded eigenvalue.
	
\noindent (1)	For $d=5$, assume $|V(x)|\lesssim (1+|x|)^{-\beta}$ with some $\beta>11$.  Let $s>5+1/2$.
	
	i) If $0$ is of the first kind resonance of $H$, we have
	\begin{equation*}\label{KJ-5-1}
	\|e^{itH}P_{ac}\|_{B(s, -s)}\lesssim (1+|t|)^{-3/4}, \,\, t\in\mathbf{R}.
	\end{equation*}
	
	ii) If $0$ is of the second  resonance or an eigenvalue of $H$, we have
	 	\begin{equation*}\label{KJ-5-2}
	 \|e^{itH}P_{ac}\|_{B(s, -s)}\lesssim (1+|t|)^{-1/4}, \,\, t\in\mathbf{R}.
	 \end{equation*}
	
	\noindent (2) For $d=6$, assume $|V(x)|\lesssim (1+|x|)^{-\beta}$ with some $\beta>14$.  Let $s>7$.
	
	 i) If $0$ is of the first kind resonance of $H$, we have
	 \begin{equation*}\label{KJ-6-1}
	 \|e^{itH}P_{ac}\|_{B(s, -s)}\lesssim (1+|t|)^{-1/2}, \,\, t\in\mathbf{R}.
	 \end{equation*}
	
	 ii) If $0$ is of the second  resonance or an eigenvalue of $H$, we have
	 \begin{equation*}\label{KJ-6-2}
	 \|e^{itH}P_{ac}\|_{B(s, -s)}\lesssim \big(\ln|t|\big)^{-1}, \,\, |t|>1,  t\in\mathbf{R}.
	 \end{equation*}

	\noindent (3) For $d=7$, assume $|V(x)|\lesssim (1+|x|)^{-\beta}$ with some $\beta>9$.  Let $s>4+1/2$.   If $0$ is a resonance  or an eigenvalue of $H$, we have
	 \begin{equation*}\label{KJ-7-1}
	 \|e^{itH}P_{ac}\|_{B(s, -s)}\lesssim (1+|t|)^{-1/4}, \,\, t\in\mathbf{R}.
	 \end{equation*}
	
\noindent	(4) For $d=8$, assume $|V(x)|\lesssim (1+|x|)^{-\beta}$ with some $\beta>8$.  Let $s>4$.    If $0$ is a resonance  or an eigenvalue of $H$, we have
	 \begin{equation*}\label{KJ-8-1}
	 \|e^{itH}P_{ac}\|_{B(s, -s)}\lesssim \big(\ln |t|\big)^{-1}, \,\,  |t|>1,  t\in\mathbf{R}.
	 \end{equation*}
	
	\noindent (5) For $d\geq9$, assume $|V(x)|\lesssim (1+|x|)^{-\beta}$ with some $\beta>d$.  Let $s>d/2$.   If $0$ is an eigenvalue of $H$, we have
	 \begin{equation*}\label{KJ-9-1}
	 \|e^{itH}P_{ac}\|_{B(s, -s)}\lesssim (1+|t|)^{-(d-8)/4}, \,\, t\in\mathbf{R}.
	 \end{equation*}
	\end{corollary}

\subsection{Further remarks}

Here, we give some further remarks about $L^p$-decay estimates.
For Schr\"odinger operator, it was well known that the $L^{1}- L^{\infty}$ decay estimate of $e^{it(-\Delta+V)}$, is very important and widely applied to the nonlinear soliton problems. For $d\geq3$, Journ\'e, Soffer and Sogge in \cite{JSS} first established the $L^{1}- L^{\infty}$ dispersive estimate in the regular case. For $d\le 3$, one can see  Rodnianski and Schlag \cite{Rod-Schlag}, Schlag and Goldberg \cite{Goldberg-Schlag}, Schlag \cite{Schlag} and so on.  Yajima in \cite{Yajima-JMSJ-95} also established the $L^{1}- L^{\infty}$  estimates for $-\Delta+V$ by wave operator method.
For many further studies, one can refer to \cite{ ES1, ES2,  Goldberg-Green-1, Goldberg-Green-2, Schlag} and therein references.

For fourth-order Schr\"odinger operator $H=(-\Delta)^2+V$, In \cite{FSY}, the first and third author and Soffer applied the Kato-Jensen decay estimates and local decay estimate to  obtain the $L^{1}(\mathbf{R}^3)- L^{\infty}(\mathbf{R}^3)$ estimate, which is not optimal. Green and Toprak  in the nice work \cite{Green-Topark} further studied the $L^{1}- L^{\infty}$ dispersive estimate for $e^{itH}$ in 4-dimension in the regular case and the cases of zero resonance. In the regular case, the time decay is natural $|t|^{-1}$. Their proof is based on the idea of \cite{FSY} and \cite{Schlag}. Naturally, it would be interesting to establish  $L^{1}(\mathbf{R}^d)-L^{\infty}(\mathbf{R}^d)$ estimate of fourth-order Schr\"odinger operator for $d\ge 5$. In particular, I remark that our asymptotic expansion of $R_V(z)$ will further plays a fundamental role in these related studies. In our forthcoming paper \cite{FSWY},  we also discuss the lower energy resolvent expansion at threshold and time decay estimates for higher order elliptic operator $P(D)+V$.

The paper is organized as follows.  In Section \ref{perturbed-resolvent}, we give the asymptotic expansions of $R_{V}(z)$ with the existence of zero-energy resonance and the classification of resonance subspaces.  In Section \ref{Kato-Jensen-estimate}, under the assumption that $H=(-\Delta)^2+V$ does not exist positive embedded eigenvalue, we establish the asymptotic expansions in time and Kato-Jensen type decay estimate for $e^{itH}$  with existence of zero-energy resonance. In the last section, we show the iterated processes of how to derive the expansion of $R_{V}(z)$ and the proof of  classification of resonance subspaces.

\section{Asymptotic expansions of $R_{V}(z)$ at zero threshold }\label{perturbed-resolvent}

For $z\in\mathbf{C}\setminus\mathbf{R}^{+}$, denote
\begin{equation*}
R_{0}(z):=\big[(-\Delta)^2-z\big]^{-1}, \,\,\,\, R_{V}(z):=\big[(-\Delta)^2+V-z\big]^{-1}.
\end{equation*}
For the free resolvent $R_{0}(z)$, we have the following identity
\begin{equation*}\label{free-resolvent-identity}
R_{0}(z)=\big[(-\Delta)^{2}-z\big]^{-1}=\frac{1}{2z^{1/2}}\big[(-\Delta-z^{1/2})^{-1}-(-\Delta+z^{1/2})^{-1}\big].
\end{equation*}
Let $z=\mu^4$ with $\mu$ in the first quadrant of the complex plane, then
\begin{equation}\label{decompasition}
R_{0}(\mu^4)=\frac{1}{2\mu^2}\bigg[ \big(-\Delta-\mu^2\big)^{-1}-\big(-\Delta -(-\mu^2)\big)^{-1}\bigg].
\end{equation}

Denote $R_{0}(\mu^{4}; x, y)$ be the kernel of $R_{0}(\mu^4)$.
By using the formula \eqref{decompasition} and the expansions of resolvent of free Laplacian $-\Delta$ (see e.g. \cite{J}),  we have the following expansions of $R_{0}(\mu^{4}; x, y)$ for $0<|\mu|\ll1$:

\begin{lemma}\label{free-expansions}
	For the free resolvent $R_{0}(\mu^{4})$, we have:
	
(1)\,\,  For $d=5$, let $s>5+1/2$, then in $B(s, -s)$:
\begin{equation*}\label{free-resolvent-5}
\begin{split}
R_{0}(\mu^4; x, y)=&a_{0}|x-y|^{-1}+\mu a_{1}|x-y|^{0}+\mu^3 a_{2}|x-y|^{2}\\
&+\mu^{4}a_{3}|x-y|^{3}+E_{1}(\mu)(x-y),
\end{split}
\end{equation*}
where $a_{1}, a_{2}\in\mathbf{C}\setminus\mathbf{R}$ and $a_{0},  a_{3}\in\mathbf{R}\setminus\{0\}$. Furthermore, $\|E_{1}(\mu)\|_{B(s, -s)}=O(\mu^5)$.

(2)\,\, For $d=6$, let $s>7$, then in $B(s, -s)$:
\begin{equation*}\label{free-resolvent-6}
\begin{split}
R_{0}(\mu^{4}; x, y)=&c_{0}|x-y|^{-2}+\mu^{2}c_{1}|x-y|^{0}+\mu^{4}\ln(\mu)c_{2}|x-y|^{4}\\
&+\mu^{4}c_{3}|x-y|^{4}+\mu^{4}c_{4}|x-y|^{4}\ln(|x-y|)+E_{2}(\mu)(x-y),
\end{split}
\end{equation*}
where $c_{1}, c_{3}\in\mathbf{C}\setminus\mathbf{R}$ and $c_{j}\in\mathbf{R}\setminus\{0\}, j=0, 2, 4$. Furthermore, $\|E_{3}(\mu)\|_{B(s, -s)}=O(\mu^{6-})$.

(3)\,\,  For $d=7$, let $s>4+1/2$, then in $B(s, -s)$:
\begin{equation*}\label{free-resolvent-7}
R_{0}(\mu^{4}; x, y)=b_{0}|x-y|^{-3}+\mu^3 b_{1} |x-y|^{0}+\mu^{4} b_{2}|x-y|^{1}+E_{2}(\mu)(x-y)
\end{equation*}
where $b_{1}\in\mathbf{C}\setminus\mathbf{R}$ and $b_{0},  b_{2}\in\mathbf{R}\setminus\{0\}$. Furthermore, $\|E_{2}(\mu)\|_{B(s, -s)}=O(\mu^5)$.

(4)\,\,  For $d=8$, let $s>4$, then in $B(s, -s)$:
\begin{equation*}\label{free-resolvent-8}
\begin{split}
R_{0}(\mu^{4}; x, y)=&d_{0}|x-y|^{-4}+\mu^{4}\ln(\mu) d_{1}|x-y|^{0}+\mu^{4}d_{2}|x-y|^{0}\\
&+\mu^{4}d_{3}\ln(|x-y|)+E_{4}(\mu)(x-y),
\end{split}
\end{equation*}
where $d_{2}\in\mathbf{C}\setminus\mathbf{R}$ and $d_{j}\in\mathbf{R}\setminus\{0\}, j=0, 1, 3$. Furthermore, $\|E_{4}(\mu)\|_{B(s, -s)}=O(\mu^{6-})$.

(5)\,\,  For $d\geq 9$ and $d$ is odd, let $s>d/2$, then in $B(s, -s)$ :
\begin{equation*}\label{free-resolvent-9-odd}
\begin{split}
R_{0}(\mu^4; x, y)=&c_{1}(d)|x-y|^{4-d}+\mu^{4}c_{2}(d)|x-y|^{8-d}+\mu^{5}\Bigg[\sum_{\ell=7, \ell\in 2\mathbf{N}\setminus 4\mathbf{N}}^{d-3}c_{\ell}(d)\\
&\cdot \mu^{\ell-7}|x-y|^{\ell+2-d}\Bigg]+\mu^{d-4} C(d)|x-y|^{0}+E_{d}(\mu)(x-y);
\end{split}
\end{equation*}

(6)\,\,  For $d\geq 9$ and $d$ is even, let $s>d/2$, then in $B(s, -s)$:
\begin{equation*}\label{free-resolvent-9-even}
\begin{split}
R_{0}(\mu^4; x, y)=&\tilde{c}_{1}(d)|x-y|^{4-d}+\mu^{4}\tilde{c}_{2}(d)|x-y|^{8-d}+\mu^{5}\Bigg[\sum_{\ell=7, \ell\in 2\mathbf{N}-1}^{d/2-2}\tilde{c}_{\ell}(d)\\
&\cdot \mu^{2\ell-7}|x-y|^{2\ell+2-d}\Bigg]+\mu^{d-4} \widetilde{C}(d)|x-y|^{0}+E_{d}(\mu)(x-y),
\end{split}
\end{equation*}
where $C(d), \widetilde{C}(d)\in\mathbf{C}\setminus\mathbf{R}$ and  $c_{j}(d), \tilde{c}_{j}(d),  c_{\ell}(d), \tilde{c}_{\ell}(d)\in \mathbf{R}\setminus\{0\}$ for all $1\leq j\leq2$ and all  $\ell$. Furthermore, $\|E_{d}(\mu)\|_{B(s, -s)}=O(\mu^{(d-4)+})$.
\end{lemma}

For the perturbed resolvent $R_{V}(\mu^4)$, we apply the following symmetric resolvent identity to derive the expansions as $\mu\rightarrow 0$.
\begin{equation}\label{symmetric-resolvent-idnetity}
R_{V}(\mu^4)=R_{0}(\mu^{4})-R_{0}(\mu^4)v\big[M(\mu)\big]^{-1}vR_{0}(\mu^4),
\end{equation}
where $v(x)=|V(x)|^{1/2}$,
\begin{equation}
M(\mu)=U+vR_{0}(\mu^{4})v
\end{equation}
with $U(x)=\rm{sign} (V(x))$.  Furthermore, let $w(x)=U(x)v(x)$, then
\begin{equation}\label{weighted-resolvent-identity}
wR_{V}(\mu^4)w=U-\big[M(\mu)\big]^{-1}.
\end{equation}
Note that, for $0<|\mu|\ll1$, the identity \eqref{weighted-resolvent-identity} shows $R_{V}(\mu^{4})$ in $B(s, -s)$ with suitable $s>0$ has the same asymptotic behaviors of $\big[M(\mu)\big]^{-1}$ in $B(0, 0)$ if $w(x)(1+|x|)^{s}\in L^{\infty}(\mathbf{R}^{d})$.

Next, we derive the asymptotic expansions of perturbed resolvent $R_{V}(\mu^{4})$. In the end, we identify the threshold spectral subspaces of each kind of resonances.  Let $P$ be the projection onto the span of $v$, i.e. $P=\|v\|_{L^{2}(\mathbf{R}^{d})}^{-1}v\langle v, \cdot\rangle$. Thus $P$ is a self-adjoint operator on $L^2$ with  $\dim(P)=1$.

\begin{definition} \label{resonance}
	Let $T_{0}=U+vG_{0}v$.
	\begin{enumerate}
		\item If $T_{0}$ is invertible on $L^{2}$, then we call zero is a regular point  of $H=(-\Delta)^2+V$.
		\item Let $S_{1}$ be the Riesz  projection onto the kernel of $T_{0}$, then $T_{0}+S_{1}$ is invertible on $L^{2}$. If $T_{0}$ is not invertible and $T_{1}:=S_{1}PS_{1}$ is invertible on $S_{1}L^2$, then we say there is a first kind of resonance at zero.
		\item Let $S_{2}$ be the Riesz projection onto the kernel of $T_{1}$, then $T_{1}+S_{2}$ is invertible on $S_{1}L^{2}$. If $T_{1}$ is not invertible and $T_{2}:=S_{2}vG_{2}vS_{2}$ is invertible on $S_{2}L^{2}$, then we say there is a second kind of resonance at zero.
		\item Let $S_{3}$ be the Riesz projection onto the kernel of $T_{2}$, then $T_{2}+S_{3}$ is invertible on $S_{2}L^{2}$. If $T_{2}$ is not invertible and $T_{3}:=S_{3}vG_{3}vS_{3}$ is invertible on $S_{3}L^{2}$, then we say there is a third kind of ``resonance" at zero.
	\end{enumerate}	
\end{definition}
\begin{remark}\label{properties-of-S-T}
(1) Noting that Definition \ref{resonance'} is equivalent to Definition \ref{resonance} actually after we identify the subspaces $S_{j}L^{2}$ for all $j$. The last kind resonance is actually eigenvalue.  See Proposition \ref{classification-5} - \ref{classification-9} below.
	\par 	(2) The projection operators $S_{j}, j=1,2,3$ are of finite rank. Indeed, $T_{0}=U+vG_{0}v$ which is a compact perturbation of the invertible operator $U$, thus by the Fredholm alternative theorem guarantees that $S_{1}$ is of finite rank. By the definition \ref{resonance}, we have $S_3\leq S_2\leq S_1$.Hhence $S_{j}, j=1,2,3$ are of finite rank.
	\par  (3) By the Definition \ref{resonance}, $S_{j}$ and $T_{j}$ satisfies:
		\begin{align*}
		S_{j+1}(T_j+S_{j+1})^{-1}=(T_j+S_{j+1})^{-1}S_{j+1} =S_{j+1},\,\,  0\leq j\leq2.
		\end{align*}
	\par (4) For $d=7, 8$, there are only  two kinds of resonances at zero. In fact, if $T_{0}$ is not invertible and $T_{1}$ is invertible, then zero is of first kind resonance. If $T_{1}$ is not invertible and $T_{2}$ is invertible, then zero is of second kind ``resonance".  For $d\geq9$, if $T_{0}$ is not invertible and $T_{1}$ is invertible, then zero is a ``resonance"  of $H$. The ``resonance" is actually eigenvalue.
\end{remark}

\subsection{Resolvent expansion for $d=5$}

In the 5-dimensional case, by Lemma \ref{free-expansions}(1),  in $B(s, -s)$ with $s>5+1/2$, we have
\begin{equation*}\label{free-resolvent-5-re}
R_{0}(\mu^4; x, y)=G_{0}+a_{1}\mu I+a_{2}\mu^{3} G_{2}+\mu^{4}G_{3}+O(\mu^5)
\end{equation*}
where
\begin{align*}
G_{0}(x,y)=a_{0}|x-y|^{-1};\,\,
G_{2}(x,y)=|x-y|^{2}; \,\, G_{3}(x,y)=a_{3}|x-y|^{3}.
\end{align*}

Next, we give the expansions of $R_{V}(\mu)$ in $B(s, -s)$ with $s>5+1/2$.
\begin{theorem}\label{RV-expansions-5}
For $d=5$ and $0<|\mu|\ll1$, assume $|V(x)|\lesssim (1+|x|)^{-\beta}$ with some $\beta>11$. Let $s>5+1/2$.  Then in $B(s, -s)$ we have:
	
	(1) If $0$ is a regular point of $H$, then
	\begin{equation*}\label{RV-5-0}
	R_{V}(\mu^4)=B_{5}^{0}+\mu\alpha_{1}B_{5}^{1}+\mu^{2}\alpha_{2}B_{5}^{2}+O(\mu^3),
	\end{equation*}
	where $B_{5}^{j}\in B(s, -s)$ are self-adjoint operators for $j=0, 1, 2$ and $\alpha_{1}, \alpha_{2}\in\mathbf{C}\setminus\mathbf{R}$.
	
	(2) If $0$ is of the first kind resonance of $H$, then
	\begin{equation*}\label{RV-5-1}
	R_{V}(\mu^4)=\frac{1}{\mu}\alpha_{-1, 1}B_{5, 1}^{-1}+\alpha_{0, 1}B_{5, 1}^{0}+\mu\alpha_{1, 1}B_{5, 1}^{1}+O(\mu^2),
	\end{equation*}
	where $B_{5, 1}^{j}\in B(s, -s)$ are self-adjoint operators and $\alpha_{-j, 1}\in\mathbf{C}\setminus\mathbf{R}$ for $j=-1, 0, 1$ .
	
	(3) If $0$ is of the second kind resonance of $H$, then
	\begin{equation*}\label{RV-5-2}
	R_{V}(\mu^4)=\frac{1}{\mu^{3}}\alpha_{-3, 2}B_{5, 2}^{-3}+\frac{1}{\mu^2}\alpha_{-2, 2}B_{5, 2}^{-2}+\frac{1}{\mu}\alpha_{-1, 2}B_{5, 2}^{-1}+\alpha_{0, 2}B_{5, 2}^{0}+O(\mu),
	\end{equation*}
	where $B_{5, 2}^{j}\in B(s, -s)$ are self-adjoint operators and $\alpha_{-j, 2}\in\mathbf{C}\setminus\mathbf{R}$ for $j=0, 1, 2, 3$ .
	
	(4) If $0$ is an eigenvalue of $H$, then
	\begin{equation*}\label{RV-5-3}
	R_{V}(\mu^4)=\frac{1}{\mu^{4}}B_{5, 3}^{-4}+\frac{1}{\mu^3}\alpha_{-3,3}B_{5, 3}^{-3}+\frac{1}{\mu^2}\alpha_{-2, 3}B_{5, 3}^{-2}+\frac{1}{\mu}\alpha_{-1, 3}B_{5, 3}^{-1}+\alpha_{0, 3}B_{5, 3}^{0}+O(\mu),
	\end{equation*}
	where $B_{5, 3}^{-j}\in B(s, -s)$ are self-adjoint operators for $0\leq j\leq4$ and $\alpha_{-j, 3}\in\mathbf{C}\setminus\mathbf{R}$ for $0\leq j\leq 3$.
	
\end{theorem}

\begin{remark}
	Noting that the weight power $s$ here we choose is to ensure that we can get the asymptotic expansions for the last kind of resonance. Actually, for the regular, first kind and second kind of resonance cases we do not need to choose $s$ so big. The reason is that we do not need to expand the free resolvent $R_{0}(\mu^{4})$ to the order $\mu^4$. The principle we choose $s$ is to ensure that the asymptotic expansion of $R_{0}(\mu^4)$ is valid in $B(s, -s)$ with $R_{0}(\mu^4)$ expanding to the order we need for each kind of resonances. For instant, for $d=5$ in the regular case we only need expand $R_{0}(\mu^4)$ to the order $\mu^{2}$, thus only need $s>2+1/2$.
\end{remark}

Now, we give the classification of the spectral subspaces of $d=5$.

\begin{proposition}\label{classification-5}
	For $d=5$, let $|V(x)|\lesssim(1+|x|)^{-\beta}$ with some $\beta>11$. Then the following statements hold:
	
	(1)  $\phi\in S_{1}L^{2}(\mathbf{R}^{5})\setminus\{0\}$ if and only if  $\phi=Uv\psi$ where $\psi\in W_{3/2}(\mathbf{R}^{5})$ satisfies $H\psi=0$ in distributional sense and
	\begin{equation*}
	\psi(x)=-\int_{\mathbf{R}^{5}}\frac{a_{0}}{|x-y|}v(y)\phi(y){\rm d}y.
	\end{equation*}
	
	 (2)  $\phi=Uv\psi\in S_{2}L^{2}(\mathbf{R}^{5})\setminus\{0\}$ if and only if $\psi\in W_{1/2}(\mathbf{R}^{5})$.
	
	 (3)  $\phi=Uv\psi\in S_{3}L^{2}(\mathbf{R}^{5})\setminus\{0\}$ if and only if $\psi\in L^{2}(\mathbf{R}^{5})$.
\end{proposition}

\subsection{Resolvent expansion for $d=6$}
In the 6-dimensional case, by Lemma \ref{free-expansions}(2),  in $B(s, -s)$ with $s>7$, we have
\begin{equation*}\label{free-resolvent-6-re}
R_{0}(\mu^4; x, y)=G_{0}+c_{1}\mu^2 I+\mu^{4}c(\mu)G_{2}+\mu^{4}G_{3}+O(\mu^{6-})
\end{equation*}
where
\begin{align*}
&G_{0}(x,y)=c_{0}|x-y|^{-2}; \,\,\,\, c(\mu)=c_{2}\ln(\mu)+c_{3};\\
&G_{2}(x,y)=|x-y|^{4}; \,\,\,\, G_{3}(x,y)=c_{4}|x-y|^{4}\ln(|x-y|).
\end{align*}

The definition of resonance for $d=6$ is the same as Definition \ref{resonance} by replacing the representations of $G_{j}$ into the corresponding case of $d=6$.   Next, we give the expansions of $R_{V}(\mu)$ in $B(s, -s)$ with $s>7$.
\begin{theorem}\label{RV-expansions-6}
For $d=6$ and $0<|\mu|\ll1$,	assume $|V(x)|\lesssim (1+|x|)^{-\beta}$ with some $\beta>14$.  Let $s>7$.   Then in $B(s, -s)$ we have:
	
	(1) If $0$ is a regular  point of $H$, then
	\begin{equation*}\label{RV-6-0}
	R_{V}(\mu^4)=B_{6}^{0}+\mu^{2}\rho_{1, 1}B_{6}^{1}+\mu^{4}\ln(\mu)B_{6}^{2}+\mu^{4}\rho_{2, 1}B_{6}^{3}+O(\mu^{6-}),
	\end{equation*}
	where $B_{6}^{j}\in B(s, -s)$ are self-adjoint operators for $0\leq j\leq 3$ and $\rho_{1}, \rho_{2}\in\mathbf{C}\setminus\mathbf{R}$.
	
	(2) If $0$ is of the first kind resonance of $H$, then
	\begin{equation*}\label{RV-6-1}
	R_{V}(\mu^4)=\frac{1}{\mu^2}\rho_{-2, 1}B_{6, 1}^{-2}+\ln(\mu)B_{6, 1}^{-1}+\rho_{0, 1}B_{6, 1}^{0}+O(\mu^{2-}),
	\end{equation*}
	where $B_{6, 1}^{-j}\in B(s, -s)$ are self-adjoint operators for $j=0, 1, 2$ and $\rho_{-2, 1}, \rho_{0, 1}\in\mathbf{C}\setminus\mathbf{R}$.
	
	(3) If $0$ is of the second kind resonance of $H$, then
	\begin{equation*}\label{RV-6-2}
	R_{V}(\mu^4)=\frac{1}{\mu^{4} c(\mu)}\rho_{-4, 2}B_{6, 2}^{-4}+\frac{1}{\mu^{4}(c(\mu))^{2}}\rho_{-3, 2}B_{6, 2}^{-3}+\frac{1}{\mu^2}\rho_{-2, 2}B_{6, 2}^{-2}+\frac{1}{\mu^2 c(\mu)}\rho_{-1, 2}B_{6, 2}^{-1}+\rho_{0, 2}B_{6, 2}^{0}+O\big((c(\mu))^{-1}\big),
	\end{equation*}
	where $B_{6, 2}^{-j}\in B(s, -s)$ are self-adjoint operators and $\rho_{-j, 2}\in\mathbf{C}\setminus\mathbf{R}$ for $0\leq j\leq4$ .
	
	(4) If $0$ is an eigenvalue of $H$, then
	\begin{equation*}\label{RV-6-3}
	\begin{split}
	R_{V}(\mu^4)=&\frac{1}{\mu^{4}}B_{6, 3}^{-5}+\frac{1}{\mu^{4}c(\mu)}\rho_{-4,3}B_{6, 3}^{-4}+\frac{1}{\mu^2}\rho_{-3, 3}B_{6, 3}^{-3}\\
	&+(c(\mu))^{2}\rho_{-2, 3}B_{6, 3}^{-2}+c(\mu)\rho_{-1, 3}B_{6, 3}^{-1}+\rho_{0, 3}B_{6, 3}^{0}+O\big((c(\mu))^{-1}\big),
	\end{split}
	\end{equation*}
	where $B_{6, 3}^{-j}\in B(s, -s)$ are self-adjoint operators for $0\leq j\leq5$ and $\rho_{-j, 3}\in\mathbf{C}\setminus\mathbf{R}$ for $0\leq j\leq 3$.
\end{theorem}

Now, we give the classification of the spectral subspaces of $d=6$.

\begin{proposition}\label{classification-6}
For $d=6$,	let $|V(x)|\lesssim(1+|x|)^{-\beta}$ with some $\beta>14$. Then  the following statements hold:
	
	(1)  $\phi\in S_{1}L^{2}(\mathbf{R}^{6})\setminus\{0\}$ if and only if  $\phi=Uv\psi$ where $\psi\in L^{2, }_{-\sigma}(\mathbf{R}^{6})$ with $\sigma\in (1,\,\, \beta-1)$ satisfies $H\psi=0$ in distributional sense and
	\begin{equation*}
	\psi(x)=-\int_{\mathbf{R}^{6}}\frac{c_{0}}{|x-y|^{2}}v(y)\phi(y){\rm d}y.
	\end{equation*}
	
	(2)  $\phi=Uv\psi\in S_{2}L^{2}(\mathbf{R}^{6})\setminus\{0\}$ if and only if $\psi\in L^{2}_{-\sigma}(\mathbf{R}^{6})$ with $\sigma\in (0,\,\, \beta-1)$.
	
	(3)  $\phi=Uv\psi\in S_{3}L^{2}(\mathbf{R}^{6})\setminus\{0\}$ if and only if $\psi\in L^{2}(\mathbf{R}^{6})$.
\end{proposition}

\subsection{Resolvent expansion for $d=7$}
In the 7-dimensional case, by Lemma \ref{free-expansions}(3),  in $B(s, -s)$ with $s>4+1/2$, we have
\begin{equation}\label{free-resolvent-7-re}
R_{0}(\mu^4; x, y)=G_{0}+b_{1}\mu^{3} I+\mu^{4} G_{2}+O(\mu^5)
\end{equation}
where
\begin{align*}
G_{0}(x,y)=b_{0}|x-y|^{-3};\,\,
G_{2}(x,y)=b_{2}|x-y|.
\end{align*}

The definition of resonance for $d=7$ is the same as Definition \ref{resonance} by replacing the representations of $G_{j}$ in the corresponding case $d=7$.  Note that, in the case $d=7$ there are only one kind of resonance. The second kind of ``resonance" is eigenvalue.

Next, we give the expansions of $R_{V}(\mu)$ in $B(s, -s)$ with $s>4+1/2$.
\begin{theorem}\label{RV-expansions-7}
 For $d=7$ and $0<|\mu|\ll1$,	assume $|V(x)|\lesssim (1+|x|)^{-\beta}$ with some $\beta>9$. Let $s>4+1/2$.  Then in $B(s, -s)$ we have:
	
	(1) If $0$ is a regular  point of $H$, then
	\begin{equation*}\label{RV-7-0}
	R_{V}(\mu^4)=B_{7}^{0}+\mu^{3}\beta_{1}B_{7}^{1}+\mu^{4}\beta_{2}B_{7}^{2}+O(\mu^5),
	\end{equation*}
	where $B_{7}^{j}\in B(s, -s)$ are self-adjoint operators for $j=0, 1, 2$ and $\beta_{1}, \beta_{2}\in\mathbf{C}\setminus\mathbf{R}$.
	
	(2) If $0$ is of the first kind resonance of $H$, then
	\begin{equation*}\label{RV-7-1}
	R_{V}(\mu^4)=\frac{1}{\mu^3}\beta_{-3, 1}B_{7, 1}^{-3}+\frac{1}{\mu^2}\beta_{-2, 1}B_{7, 1}^{-2}+\frac{1}{\mu}\beta_{-1, 1}B_{7, 1}^{-1}+\beta_{0, 1}B_{7, 1}^{0}+O(\mu),
	\end{equation*}
	where $B_{7, 1}^{-j}\in B(s, -s)$ are self-adjoint operators and $\beta_{-j, 1}\in\mathbf{C}\setminus\mathbf{R}$ for $0\leq j\leq3$.
	
	(3) If $0$ is an eigenvalue of $H$, then
	\begin{equation*}\label{RV-7-2}
	R_{V}(\mu^4)=\frac{1}{\mu^{4}}B_{7, 2}^{-4}+\frac{1}{\mu^{3}}\beta_{-3, 2}B_{7, 2}^{-3}+\frac{1}{\mu^2}\beta_{-2, 2}B_{7, 2}^{-2}+\frac{1}{\mu}\beta_{-1, 2}B_{7, 2}^{-1}+\beta_{0, 2}B_{7, 2}^{0}+O(\mu),
	\end{equation*}
	where $B_{7, 2}^{-4},  B_{7, 2}^{-j}\in B(s, -s)$ are self-adjoint operators and $\beta_{-j, 2}\in\mathbf{C}\setminus\mathbf{R}$ for $0\leq j\leq 3$ .
\end{theorem}

Now, we give the classification of the spectral subspaces of $d=7$.

\begin{proposition}\label{classification-7}
For $d=7$, let $|V(x)|\lesssim(1+|x|)^{-\beta}$ with some $\beta>9$.   The following statements hold:
	
	(1) $\phi\in S_{1}L^{2}(\mathbf{R}^{7})\setminus\{0\}$ if and only if  $\phi=Uv\psi$ where $\psi\in W_{1/2}(\mathbf{R}^{7})$  satisfies $H\psi=0$ in distributional sense and
	\begin{equation*}
	\psi(x)=-\int_{\mathbf{R}^{5}}\frac{b_{0}}{|x-y|^{3}}v(y)\phi(y){\rm d}y.
	\end{equation*}
	
	(2)  $\phi=Uv\psi\in S_{2}L^{2}(\mathbf{R}^{7})\setminus\{0\}$ if and only if $\psi\in L^{2}(\mathbf{R}^{7})$.
\end{proposition}

\subsection{Resolvent expansion for $d=8$}
In the 8-dimensional case, by Lemma \ref{free-expansions}(4),  in $B(s, -s)$ with $s>4$, we have
\begin{equation*}\label{free-resolvent-8-re}
R_{0}(\mu^4; x, y)=G_{0}+\mu^{4}d(\mu) I+\mu^{4} G_{2}+O(\mu^{6-})
\end{equation*}
where
\begin{align*}
G_{0}(x,y)=d_{0}|x-y|^{-4};\,\,\, d(\mu)=d_{1}\ln\mu+d_{2};\,\,\, G_{2}(x,y)=d_{3}\ln(|x-y|).
\end{align*}

The definition of resonance for $d=8$ is the same as Definition \ref{resonance} by replacing the representations of $G_{j}$ in the corresponding case $d=8$.  Note that, in the case $d=8$ there are only one kind of resonance. The second kind of ``resonance" is eigenvalue.    Next, we give the expansions of $R_{V}(\mu)$ in $B(s, -s)$ with $s>4$.
\begin{theorem}\label{RV-expansions-8}
For $d=8$ and $0<|\mu|\ll1$,	assume $|V(x)|\lesssim (1+|x|)^{-\beta}$ with some $\beta>8$.  Let $s>4$.   Then in $B(s, -s)$ we have:
	
	(1) If $0$ is a regular  point of $H$, then
	\begin{equation*}\label{RV-8-0}
	R_{V}(\mu^4)=B_{8}^{0}+\mu^{4}\ln(\mu) B_{8}^{1}+\mu^{4}\varrho B_{8}^{2}+O(\mu^{6-}),
	\end{equation*}
	where $B_{8}^{j}\in B(s, -s)$ are self-adjoint operators for $j=0, 1, 2$ and $\varrho\in\mathbf{C}\setminus\mathbf{R}$.
	
	(2) If $0$ is of the first kind resonance of $H$, then
	\begin{equation*}\label{RV-8-1}
	R_{V}(\mu^4)=\frac{1}{\mu^4 d(\mu)}\varrho_{-3, 2}B_{8, 2}^{-3}+\frac{1}{\mu^4(d(\mu))^2}\varrho_{-2, 2}B_{8, 2}^{-2}+\frac{1}{\mu^2(d(\mu))^2}\varrho_{-1, 2}B_{8, 2}^{-1}+\varrho_{0, 2}B_{8, 2}^{0}+O\big((d(\mu))^{-1}\big),
	\end{equation*}
	where $B_{8, 2}^{-j}\in B(s, -s)$ are self-adjoint operators and $\varrho_{-j, 2}\in\mathbf{C}\setminus\mathbf{R}$ for $0\leq j\leq3$.
	
	(3) If $0$ is an eigenvalue of $H$, then
	\begin{equation*}\label{RV-8-2}
	R_{V}(\mu^4)=\frac{1}{\mu^{4}}B_{8, 3}^{-4}+\frac{1}{\mu^{4} d(\mu)}\varrho_{-3, 3}B_{8, 3}^{-3}+\frac{1}{\mu^2}\varrho_{-2, 3}B_{8, 3}^{-2}+\frac{1}{\mu^2 d(\mu)}\varrho_{-1, 3}B_{8, 3}^{-1}+\varrho_{0, 3}B_{8, 3}^{0}+O\big(d(\mu)^{-1}\big),
	\end{equation*}
	where $B_{8, 3}^{-4},  B_{8, 3}^{-j}\in B(s, -s)$ are self-adjoint operators and $\varrho_{-j, 3}\in\mathbf{C}\setminus\mathbf{R}$ for $0\leq j\leq 3$ .
	
\end{theorem}

Now, we give the classification of the spectral subspaces of $d=8$.

\begin{proposition}\label{classification-8}
For $d=8$,	let $|V(x)|\lesssim(1+|x|)^{-\beta}$ with some $\beta>8$.  The following statements hold:
	
	(1)  $\phi\in S_{1}L^{2}(\mathbf{R}^{8})\setminus\{0\}$ if and only if $\phi=Uv\psi$ where $\psi\in W_{0}(\mathbf{R}^{8})$  satisfies $H\psi=0$ in distributional sense and
	\begin{equation*}
	\psi(x)=-\int_{\mathbf{R}^{8}}\frac{d_{0}}{|x-y|^{4}}v(y)\phi(y){\rm d}y.
	\end{equation*}
	
	(2)  $\phi=Uv\psi\in S_{2}L^{2}(\mathbf{R}^{8})\setminus\{0\}$ if and only if $\psi\in L^{2}(\mathbf{R}^{8})$.
	
\end{proposition}

\subsection{Resolvent expansion for $d\geq9$}
In the cases $d\geq 9$, by Lemma \ref{free-expansions}, in $B(s, -s)$ with $s>d/2$, we have
\begin{equation*}\label{free-resolvent-9}
R_{0}(\mu^4; x, y)=G_{0}+\mu^{4}G_{1}+\mu^{5}G(\mu)+\mu^{d-4} c_{0}(d)G_{2}+O(\mu^{(d-4)+}),
\end{equation*}
where
\begin{align*}
&G_{0}(x, y)
=\begin{cases}
c_{1}(d)|x-y|^{4-d}, &  d\geq 9\quad odd;\\ c_{2}(d)|x-y|^{4-d}, &  d\geq9 \quad even;
\end{cases}\\
&G_{1}(x, y)=
\begin{cases}
c_{3}(d)|x-y|^{8-d}, &  d\geq 9\quad odd;\\ c_{4}(d)|x-y|^{8-d}, &  d\geq9 \quad even;
\end{cases}\\
&G(\mu)(x,y)=
\begin{cases} \sum_{\ell=7, \ell\in 2\mathbf{N}\setminus 4\mathbf{N}}^{d-3}c_{\ell}(d)\mu^{\ell-7}|x-y|^{\ell+2-d}, & d\geq 9\quad odd;\\  \\ \sum_{\ell=7, \ell\in 2\mathbf{N}-1}^{d/2-2}\mu^{2j-7}\tilde{c}_{\ell}(d)|x-y|^{2\ell+2-d}, &  d\geq9 \quad even;
\end{cases}\\
&G_{2}(x, y)=
\begin{cases}
c_{5}(d)|x-y|^{0}, &  d\geq 9\quad odd;\\ c_{6}(d)|x-y|^{0}, &  d\geq9 \quad even;
\end{cases}
\end{align*}
where $c_{0}(d)\in\mathbf{C}\setminus\mathbf{R}$ and  $c_{j}(d), c_{\ell}(d), \tilde{c}_{\ell}(d)\in \mathbf{R}\setminus\{0\}$ for all $1\leq j\leq6$ and all  $\ell$.

Note that, for $d\geq9$, zero is either a regular  point or an eigenvalue of $H$.  Next, we give the expansions of $R_{V}(\mu)$ in $B(s, -s)$ with $s>d/2$.
\begin{theorem}\label{RV-expansions-d}
	For $d\geq9$ and $0<|\mu|\ll1$, assume $|V(x)|\lesssim (1+|x|)^{-\beta}$ with some $\beta>d$. Let $s>d/2$,   then in $B(s, -s)$ we have:
	
	(1) If $0$ is a regular  point of $H$, then
	\begin{equation*}\label{RV-d-0}
	R_{V}(\mu^4)=B_{d}^{0}+\mu^{4}B_{d}^{1}+\mu^{5}B_{d}^{2}+\kappa(d)\mu^{d-4}B_{d}^{3}+O(\mu^{d-4+}),
	\end{equation*}
	where $B_{d}^{j}\in B(s, -s)$ are self-adjoint operators for $j=0, 1, 2, 3$ and $\kappa(d)\in\mathbf{C}\setminus\mathbf{R}$.
	
	(2) If $0$ is an eigenvalue of $H$, then
	\begin{equation*}\label{RV-d-1}
	R_{V}(\mu^4)=\frac{1}{\mu^4}B_{d, 1}^{-3}+\frac{1}{\mu^3}B_{d, 1}^{-2}+\mu^{d-12}\kappa_{-1, 1}B_{d, 1}^{-1}+\kappa_{0, 1}B_{d, 1}^{0}+O(\mu^{d-8}),
	\end{equation*}
	where $B_{d, 1}^{-j}\in B(s, -s)$ are self-adjoint operators for $0\leq j\leq3$  and $\kappa_{-1, 1},  \kappa_{0, 1} \in\mathbf{C}\setminus\mathbf{R}$.
	
\end{theorem}

Now, we give the classification of the spectral subspaces of $d\geq9$.

\begin{proposition}\label{classification-9}
	 For $d\geq9$,  $\phi\in S_{1}L^{2}(\mathbf{R}^{d})\setminus\{0\}$ if and only if $\phi=Uv\psi$ where $\psi\in L^{2 }(\mathbf{R}^{d})$ satisfies $H\psi=0$ in distributional sense and
	\begin{equation*}
	\psi(x)=-\int_{\mathbf{R}^{d}}\frac{c(d)}{|x-y|^{d-4}}v(y)\phi(y){\rm d}y.
	\end{equation*}
	
\end{proposition}

\section{Kato-Jensen type decay estimates}\label{Kato-Jensen-estimate}
In \cite{FSY}, Feng, Soffer and Yao gave the Kato-Jensen type decay estimates for $e^{itH}P_{ac}$ with the assumption that zero is a regular point of $H$ with $d=3$ and $d\geq5$. For $d\geq5$ and $s>\beta(d)>0$ with suitable constant $\beta(d)$, the results in \cite{FSY} tells that
\begin{equation*}
\|e^{itH}P_{ac}\|_{B(s, -s)}\lesssim (1+|t|)^{-d/4}, \,\, t\in\mathbf{R}.
\end{equation*}
In this section, we show the Kato-Jensen type decay estimates for $e^{itH}$ with the existence of zero-resonance with $d\geq5$.  By the spectral represent theorem,
\begin{equation}\label{spectral}
e^{itH}P_{ac}=\int_{0}^{\infty}e^{it\lambda^4}4\lambda^3{\rm d}E(\lambda).
\end{equation}

For the spectral density ${\rm d}E(\lambda)$, from \cite[Theorem 2.23] {FSY}, for $|V(x)|\lesssim (1+|x|)^{-\beta(d)-}$ and $s>k+1/2$ with $k=0, 1, 2$, $\beta(d)$ is given constant as following depends on dimension $d$,  we know that the limiting absorption principle hold for $R_{V}(z)$, i.e. in $B(s, -s)$ with $s>k+1/2$, the limits exist
\begin{equation*}
R_{v}(\lambda\pm i0)=\lim_{\epsilon\downarrow 0} R_{V}(\lambda\pm i\epsilon).
\end{equation*}
Furthermore, for high energy we have
\begin{equation}\label{highenergy}
\Big\|\frac{{\rm d}^{k+1}}{{\rm d}\lambda^{k+1}}E(\lambda)\Big\|_{B(s, -s)}=O(\lambda^{-(3+3k)}), \,\,\lambda\rightarrow\infty.
\end{equation}
Hence, by Stone's formula and Theorem \ref{RV-expansions-5} --  \ref{RV-expansions-d}, for lower energy part we have:
\begin{proposition}
	For $H=(-\Delta)^2+V$, $dE(\lambda)$ be the spectral density of $H$. For $0<\lambda\ll1$, we have:
	
(1)	For $d=5$, assume $|V(x)|\lesssim (1+|x|)^{-\beta}$ with some $\beta>11$.  Let $s>5+1/2$,   then in $B(s, -s)$:
	
	i) If $0$ is of the first kind resonance of $H$, we have
	\begin{equation*}\label{E-5-1}
	\pi {\rm d}E(\lambda)=\lambda^{-1}{\rm Im}(\alpha_{-1, 1}) B_{5, 1}^{-1}+{\rm Im}(\alpha_{0, 1})B_{5, 1}^{0}+\lambda{\rm Im}(\alpha_{1, 1})B_{5, 1}^{1}+O(\lambda^2) .
	\end{equation*}
	
	ii) If $0$ is of the second  kind resonance or an eigenvalue of $H$ with  $k=2, 3$ respectively,  we have
	\begin{equation*}\label{E-5-2}
	\pi {\rm d}E(\lambda)=\lambda^{-3}{\rm Im}(\alpha_{-3, k})B_{5, k}^{-3}+\lambda^{-2}{\rm Im}(\alpha_{-2, k})B_{5, k}^{-2}+\lambda^{-1}{\rm Im}(\alpha_{-1, k})B_{5, k}^{-1}+{\rm Im}(\alpha_{0, k})B_{5, k}^{0}+O(\lambda).
	\end{equation*}
	
(2)	For $d=6$, assume $|V(x)|\lesssim (1+|x|)^{-\beta}$ with some $\beta>14$.  Let $s>7$,   then in $B(s, -s)$:
	
	i) If $0$ is of the first kind resonance of $H$, we have
	\begin{equation*}\label{E-6-1}
   \pi {\rm d}E(\lambda)=\frac{1}{\lambda^2}{\rm Im}(\rho_{-2, 1})B_{6, 1}^{-2}+\ln(\lambda){\rm Im}(\rho_{-1, 1})B_{6, 1}^{-1}+{\rm Im}(\rho_{0, 1})B_{6, 1}^{0}+O(\lambda^{2-}),
	\end{equation*}
	
	ii) If $0$ is of the second kind resonance or an eigenvalue of $H$ with $k=2, 3$ respectively, we have
	\begin{equation*}\label{E-6-2}
	\begin{split}
\pi {\rm d}E(\lambda)=&\frac{1}{\lambda^{4}\Big([c_{2}\ln\lambda-{\rm Im}(c_{3})]^2+(c_{2}\pi)^2\Big)}{\rm Im}(\rho_{-4, k})B_{6, k}^{-4}+\frac{1}{\lambda^{4}(\ln\lambda)^{2}}{\rm Im}(\rho_{-3, k})B_{6, k}^{-3}\\
&+\frac{1}{\lambda^2}{\rm Im}(\rho_{-2, k})B_{6, k}^{-2}+\frac{1}{\lambda^2\Big([c_{2}\ln\lambda-{\rm Im}(c_{3})]^2
	+(c_{2}\pi)^2\Big)}{\rm Im}(\rho_{-1, k})B_{6, k}^{-1}\\
&+{\rm Im}(\rho_{0, k})B_{6, k}^{0}+O\big((\ln\lambda)^{-1}\big).
\end{split}
	\end{equation*}

(3)	For $d=7$, assume $|V(x)|\lesssim (1+|x|)^{-\beta}$ with some $\beta >9$.  Let $s>4+1/2$.  If $0$ is a resonance or  an eigenvalue of $H$ with $k=1, 2$ respectively,   then in $B(s, -s)$ we have
	\begin{equation*}\label{E-7-1}
	\pi {\rm d}E(\lambda)=\lambda^{-3}{\rm Im}(\beta_{-3, k})B_{7, k}^{-3}+\lambda^{-2}{\rm Im}(\beta_{-2, k})B_{7, k}^{-2}
	+\lambda^{-1}{\rm Im}(\beta_{-1, k})B_{7, k}^{-1}+{\rm Im}(\beta_{0, k})B_{7, k}^{0}+O(\lambda).
	\end{equation*}
	
(4)	For $d=8$, assume $|V(x)|\lesssim (1+|x|)^{-\beta}$ with some $\beta>8$.  Let $s>4$.    If $0$ is a resonance or an eigenvalue of $H$ with $k=1, 2$ respectively,  then in $B(s, -s)$ we have
	\begin{equation*}\label{E-8-1}
	\begin{split}
	\pi {\rm d}E(\lambda)=&\frac{1}{\lambda^4 \Big([d_{1}\ln\lambda-{\rm Im}(d_{2})]^2+(d_{1}\pi)^2\Big)}{\rm Im}(\varrho_{-3, k})B_{8, k}^{-3}+\frac{1}{\lambda^4(\ln\lambda)^2}{\rm Im}(\varrho_{-2, k})B_{8, k}^{-2}\\
	&+\frac{1}{\lambda^2(\ln\lambda)^2}{\rm Im}(\varrho_{-1, k})B_{8, k}^{-1}+{\rm Im}(\varrho_{0, k})B_{8, k}^{0}+O\big((\ln\lambda)^{-1}\big).
	\end{split}
	\end{equation*}
		
(5) 	For $d\geq9$, assume $|V(x)|\lesssim (1+|x|)^{-\beta}$ with some $\beta>d$.  Let $s>d/2$.   If $0$ is an eigenvalue of $H$,  then in $B(s, -s)$ we have
	\begin{equation*}\label{E-9-1}
	\pi {\rm d}E(\lambda)=\lambda^{d-12}{\rm Im}(\kappa_{-1, 1})B_{d, 1}^{-1}+{\rm Im}(\kappa_{0, 1})B_{d, 1}^{0}+O(\lambda^{d-8}).
	\end{equation*}
	
\end{proposition}

With the help of the properties of $dE(\lambda)$, following the same argument for $e^{itH}P_{ac}$ as in \cite{FSY} for the regular cases, then we show the proof of Theorem \ref{eitH}.

\begin{proof}[\bf Proof of Theorem \ref{eitH}]
	Let $\chi(\lambda)$ be a smooth cutoff function with support in $0<\lambda<1$. Then
	\begin{equation*}
	\begin{split}
	e^{itH}P_{ac}=&4\int_{0}^{\infty}e^{it\lambda^4}\lambda^3 E'(\lambda){\rm d\lambda}\\
	=&4\int_{0}^{\infty}e^{it\lambda^4}\lambda^3 \chi(\lambda)E'(\lambda){\rm d\lambda}+4\int_{0}^{\infty}e^{it\lambda^4}[1-\chi(\lambda)]\lambda^3 E'(\lambda){\rm d\lambda}\\
	:=I+II.
	\end{split}
	\end{equation*}

For term $II$, note that the integral is supported in $[1, \infty)$. By \eqref{highenergy} we know, it is actually the Fourier transform of the $L^{1}(\mathbf{R})$ integrable function $[1-\chi(\lambda)]\lambda^3 E'(\lambda)$. Thus, by the Riemann-Lebesgue lemma, we know the contribution of term $II$ is $|t|^{-k}$ for any large $k>0$. 	

For term $I$, we have:  for $f(x)\in C_{0}^{\infty}(\mathbf{R})$,
\begin{equation}
\int_{0}^{\infty}e^{i\lambda x}f(x)x^{\tau}{\rm dx}\sim \sum_{j=0}\theta_{j}\lambda^{-j-1-\tau}
, \,\, {\rm Re(\tau)>-1},
\end{equation}
where $\theta_{j}=i^{j+\tau+1}\frac{j!}{\Gamma(j+1+\tau)}f^{(j)}(0)$. See Stein's book \cite[P. 355]{Stein}.

On the other hand, for $d=6, 8$,  we need to treat such term
$$\int_{0}^{\infty}\frac{e^{it\lambda}}{\lambda[(\ln\lambda-a)^2+\pi^2]}{\rm d\lambda}, \,\, a\neq 0.$$
However, the integral is $O(1/\ln t)$ as $t\rightarrow\infty$. See \cite{J1}.
\end{proof}

\section{Proof of asymptotic expansions and identification of resonance spaces}\label{proof}

By the symmetric resolvent identity \eqref{symmetric-resolvent-idnetity}, we need to derive the asymptotic expansions of $\big[M(\mu)\big]^{-1}$ in $L^{2}(\mathbf{R}^{d})$ for $\mu$ near zero. In this section, we show the processes of deriving the expansion of $\big[M(\mu)\big]^{-1}$ and identify the resonance subspaces case by case.   Here we give the needed basic lemmas in the following.
\begin{lemma}\label{Feshbach-formula}
	Let $A$ be a closed operator and $S$ a projection. Suppose $A+S$ has a bounded inverse, then $A$ has a bounded inverse if and only if
	\begin{equation}
	B\equiv S-S(A+S)^{-1}S
	\end{equation}
	has a bounded inverse in $S\mathcal{H}$. Furthermore,
	\begin{equation}
	A^{-1}=(A+S)^{-1}+(A+S)^{-1}SB^{-1}S(A+S)^{-1}.
	\end{equation}
\end{lemma}

\begin{lemma}\label{inverse formula}
	(\cite[Proposition 1]{JN2}) Let $F\subset\mathbf{C}$ have zero as an accumulation point. Let $T(z), z\in F$ be a family of bounded operators of the form
	\begin{equation*}
	T(z)=T_{0}+zT_{1}(z)
	\end{equation*}
	with $T_{1}(z)$ uniformly bounded as $z\rightarrow0$. Suppose $0$ is an isolated point of the spectrum of $T_{0}$, and let $S$ be the corresponding Riesz projection. If $T_{0}S=0$, then for sufficient small $z\in F$ the operator $\widetilde T(z):S\mathcal{H}\rightarrow S\mathcal{H}$ defined by
	\begin{equation}\label{equ:61}
	\widetilde{T}(z)=\frac{1}{z}(S-S(T(z)+S)^{-1}S)=\sum_{j=0}^{\infty}(-1)^{j}z^{j}S[T_{1}(z)(T_{0}+S)^{-1}]^{j+1}S
	\end{equation}
	is uniformly bounded as $z\rightarrow0$. The operator $T(z)$ has a bounded inverse in $\mathcal{H}$ if and only if $\widetilde{T}(z)$ has a bounded inverse in $S\mathcal{H}$, and in this case
	\begin{equation}\label{equ:62}
	T(z)^{-1}=(T(z)+S)^{-1}+\frac{1}{z}(T(z)+S)^{-1}S\widetilde{T}(z)^{-1}S(T(z)+S)^{-1}.
	\end{equation}
\end{lemma}

\begin{lemma}\label{Reisz-potential-boundedness}
	(\cite[Lemma 2.3]{J}) Denote the Riesz potential $I_{\alpha}=(-\Delta)^{-\alpha/2}$ on $\mathbf{R}^{d}$ with $0<\alpha<d$.
	\begin{enumerate}
		\item If  $0<\alpha<d/2$, $s, s'\geq0$ and $s+s'\geq\alpha$, then
		$I_{\alpha}\in B(s,\,\, -s').$
		\item If  $d/2\leq\alpha<d$, $s, s'>\alpha-d/2$ and $s+s'\geq\alpha$, then
		$I_{\alpha}\in B(s,\,\, -s').$
	\end{enumerate}
\end{lemma}

In the following, we denote $D_{j}=\Big(T_{j}+S_{j+1}\Big)^{-1}, 0\leq j\leq3$ where $T_{j}$ and $S_{j}$ see Definition \ref{resonance}.
\subsection{Proof of resolvent asymptotic expansions} In this subsection, we give the proof of the resolvent asymptotic expansions at zero-resonance case by case. Here, we give the proof details of $d=5$ and $d=6$.  The case of $d=7(d=8)$ follows by the similar process of $d=5(d=6)$ respectively. Noting that there are only two kinds of resonances for $d=7, 8$ while there are three kinds for $d=5, 6$. For $d\geq9$, from Proposition \ref{classification-9}, we know that zero is only possible to be an eigenvalue of $H$.

{\bf Proof of Theorem \ref{RV-expansions-5} ($d=5$).}
 In the case $d=5$,  for $|v(x)|\lesssim (1+|x|)^{-\beta/2}$ with some $\beta>11$, then in $B(0, 0)$ we have
\begin{equation}
M(\mu)=U+vG_{0}v+\mu a_{1}\|V\|_{L^{1}}P+\mu^{3}a_{2}vG_{2}v+\mu^{4}vG_{3}v+O(\mu^{5}).
\end{equation}
Let $a=a_{1}\|V\|_{L^{1}}$ and
\begin{align*}
\mathscr{M}(\mu)=a\mu P+a_{2}\mu^{3} vG_{2}v+\mu^{4}vG_{3}v+vE_{0}(\mu)v.
\end{align*}
Theorem \ref{RV-expansions-5} holds by symmetric resolvent identity \eqref{symmetric-resolvent-idnetity} and the following   Lemmas \ref{regular} - \ref{stop}.

\begin{lemma}\label{regular}
Assume $|v(x)|\lesssim (1+|x|)^{-\beta/2}$ with some $\beta>11$.  If zero is a regular point  of $H$, then for $0<|\mu|\ll 1$,
\begin{equation}\label{regular-expansion}
\big[M(\mu)\big]^{-1}=T_{0}^{-1}-a\mu T_{0}^{-1}PT_{0}^{-1}+a^{2}\mu^{2}T_{0}^{-1}PT_{0}^{-1}PT_{0}^{-1}+O(\mu^{3}).
\end{equation}
\end{lemma}

\begin{proof}
Since $M(\mu)=T_{0}+\mathscr{M}(\mu)$,
and $T_{0}$ is invertible,  then $\mathscr{M}(\mu)T_{0}^{-1}$ is uniformly bounded with bound less than 1 for tiny  $|\mu|$. Thus
\begin{equation}
\begin{split}
\big[M(\mu)\big]^{-1}=&T_{0}^{-1}\big[1+\mathscr{M}(\mu)T_{0}^{-1}\big]^{-1}\\
=&T_{0}^{-1}-T_{0}^{-1}\mathscr{M}(\mu)T_{0}^{-1}+T_{0}^{-1}\big(\mathscr{M}(\mu)T_{0}^{-1}\big)^{2}\big[1+\mathscr{M}(\mu)T_{0}^{-1}\big]^{-1}.
\end{split}
\end{equation}
\end{proof}

\begin{lemma}
	Assume $|v(x)|\lesssim (1+|x|)^{-\beta/2}$ with some $\beta>11$.  If zero is  of the first kind resonance,  then for $0<|\mu|\ll1$, we have
	\begin{equation}
	\begin{split}
	\big[M(\mu)\big]^{-1}=\frac{a}{\mu}S_{1}\big(S_{1}PS_{1}\big)^{-1}S_{1}+B_{0}^{1}+O(\mu),
	\end{split}
	\end{equation}
	where
	\begin{align*}
	B_{0}^{1}=S_{1}(S_{1}PS_{1})^{-1}(S_{1}PD_{0}PS_{1})(S_{1}PS_{1})^{-1}S_{1}-D_{0}PS_{1}(S_{1}PS_{1})^{-1}S_{1}-S_{1}(S_{1}PS_{1})^{-1}S_{1}PD_{0}.
	\end{align*}
\end{lemma}

\begin{proof}	
	Under the assumption on $v(x)$, then the Hilbert-Schmidt norm $\|vG_{0}v\|_{HS}$ is bounded.  Since $T_{0}=U+vG_{0}v$ which is the compact perturbation of $U$, and $\sigma(U)=\{-1, 1\}$, thus  $\sigma(T_{0})\setminus\{-1, 1\}$ is a discrete set.
	
In the first kind resonance case, in order to apply Lemma \ref{inverse formula} we need to ensure $M_{1}(\mu)$ is invertible.  We claim that if $S_{1}PS_{1}$ is invertible on $S_{1}L^{2}$, then  $M_{1}(\mu)$ is invertible on $S_{1}L^{2}$.  Indeed,
 \begin{equation*}
\begin{split}
M_{1}(\mu)=&\frac{1}{\mu}\bigg[S_{1}-S_{1}\Big(T_{0}+S_{1}+\mathscr{M}(\mu)\Big)^{-1}S_{1}\bigg]\\
=&\frac{1}{\mu}S_{1}\Bigg\{(T_{0}+S_{1})^{-1}\bigg[\mathscr{M}(\mu)(T_{0}+S_{1})^{-1}-\Big(\mathscr{M}(\mu)(T_{0}+S_{1})^{-1}\Big)^{2}\\
&+\Big(\mathscr{M}(\mu)(T_{0}+S_{1})^{-1}\Big)^{3}-\Big(\mathscr{M}(\mu)(T_{0}+S_{1})^{-1}\Big)^{4}\bigg]\Bigg\}S_{1}+O(\mu^{4})\\
=&aS_{1}PS_{1}-a^{2}\mu S_{1}PD_{0}PS_{1}+\mu^{2}\bigg[a_{2}S_{1}vG_{2}vS_{1}+S_{1}PD_{0}PD_{0}PS_{1}\bigg] \\
&+\mu^{3}\bigg[S_{1}vG_{3}vS_{1}-aa_{2}\Big(S_{1}PD_{0}vG_{2}vS_{1}+S_{1}vG_{2}vD_{0}PS_{1}\Big)-a^{4}S_{1}PD_{0}PD_{0}PD_{0}PS_{1}\bigg]+O(\mu^{4})\\
:=&aS_{1}PS_{1}+\mathscr{M}_{1}(\mu).
\end{split}
\end{equation*}
Hence, $\big[M(\mu)\big]^{-1}$ can be expanded into Neumann series provided $S_{1}PS_{1}$ is invertible on $S_{1}L^{2}$.

From Lemma \ref{inverse formula}, we have
\begin{equation}\label{M}
\big[M(\mu)\big]^{-1}=\big[M(\mu)+S_{1}\big]^{-1}+\frac{1}{\mu}\big[M(\mu)+S_{1}\big]^{-1}S_{1}\big[M(\mu)\big]^{-1}S_{1}\big[M(\mu)+S_{1}\big]^{-1}.
\end{equation}
Expanding $\big[M(\mu)+S_{1}\big]^{-1}$ and $\big[M_{1}(\mu)\big]^{-1}$ into Neumann series, then
\begin{equation*}
\big[M(\mu)\big]^{-1}=\frac{1}{\mu}S_{1}\big(aS_{1}PS_{1}\big)^{-1}S_{1}+B_{0}^{1}+O(\mu)
\end{equation*}
where $B_{0}^{1}$ is a Hilbert-Schmidt operator and
\begin{equation*}
B_{0}^{1}=S_{1}(S_{1}PS_{1})^{-1}(S_{1}PD_{0}PS_{1})(S_{1}PS_{1})^{-1}S_{1}-D_{0}PS_{1}(S_{1}PS_{1})^{-1}S_{1}-S_{1}(S_{1}PS_{1})^{-1}S_{1}PD_{0}.
\end{equation*}
\end{proof}

\begin{lemma}
	Assume $|v(x)|\lesssim (1+|x|)^{-\beta/2}$ with some $\beta>11$.  If zero is of the second kind resonance, then for $<|\mu|\ll1$, we have
	\begin{equation}\label{first}
	\begin{split}
	\big[M(\mu)\big]^{-1}=&\mu^{-3}\frac{1}{a_{2}}S_{2}\big(S_{2}vG_{2}vS_{2}\big)^{-1}S_{2}+\mu^{-2}B^{2}_{-2}+\mu^{-1}B^{2}_{ -1}+B^{2}_{0}+O(\mu)
	\end{split}
	\end{equation}
	where $B^{2}_{-2}, B^{2}_{-1}, B^{2}_{0}$ are Hilbert-Schmidt operators.
\end{lemma}

\begin{proof}
	In the second kind resonance case, we aim to derive the asymptotic expansion of $\big[M(\mu)\big]^{-1}$.  Noting that $P$ is a projection on $L^{2}$, thus the spectrum $\sigma(P)=\{0, 1\}$. Since $S_{1}$ is of finite rank, then 0 is an isolated spectrum point of $S_{1}PS_{1}$.
	$S_{2}$ is the projection onto  $\ker(S_{1}PS_{1})$, thus $S_{2}P=PS_{2}=0$ on $S_{1}L^{2}$.
	
	In order to apply Lemma \ref{inverse formula} to $\frac{1}{a}M_{1}(\mu)$, we need to check $M_{2}(\mu)$ is invertible on $S_2L^{2}$ with
	\begin{equation*}
	\begin{split}
	M_{2}(\mu)=&\frac{1}{\mu}\Big[S_{2}-S_{2}\Big(\frac{1}{a}M_{1}(\mu)+S_{1}\Big)^{-1}S_{2}\Big]
	=\frac{1}{\mu}\Big[S_{2}-S_{2}\Big(T_{1}+S_{2}+\frac{1}{a}\mathscr{M}_{1}(\mu)\Big)^{-1}S_{2}\Big]\\
	=&\frac{1}{\mu}\bigg[S_{2}-S_{2}\Big(1+\frac{1}{a}\mathscr{M}_{1}(\mu)D_{1}\Big)^{-1}S_{2}\bigg]\\
	=&\mu \frac{a_{2}}{a}S_{2}vG_{2}vS_{2}+\mu^{2} \frac{1}{a}S_{2}vG_{3}vS_{2}-\mu^{3}\frac{a_{2}^{2}}{a^{2}}S_{2}vG_{2}vS_{1}D_{1}S_{1}vG_{2}vS_{2}+O(\mu^{4})\\
	=&\mu \frac{a_{2}}{a}\bigg[S_{2}vG_{2}vS_{2}+ \frac{\mu}{a_{2}}S_{2}vG_{3}vS_{2}-\mu^{2}\frac{a_{2}}{a}S_{2}vG_{2}vS_{1}D_{1}S_{1}vG_{2}vS_{2}+O(\mu^{3})\bigg]\\
	:=&\mu \frac{a_{2}}{a}\bigg[S_{2}vG_{2}vS_{2}+\mathscr{M}_{2}(\mu)\bigg].
	\end{split}
	\end{equation*}
	Here, we need use $S_{2}P=PS_{2}=0$ and $S_{2}\leq S_{1}$. Hence,
	\begin{equation*}
	\begin{split}
	\big[M_{2}(\mu)\big]^{-1}=&\frac{1}{\mu}\frac{a}{a_{2}}\big[S_{2}vG_{2}vS_{2}+ \mathscr{M}_{2}(\mu)\big]^{-1}\\
	=&\frac{1}{\mu}\frac{a}{a_{2}}\big(S_{2}vG_{2}vS_{2}\big)^{-1}-\frac{a}{a_{2}}\big(S_{2}vG_{2}vS_{2}\big)^{-1}\big(S_{2}vG_{3}vS_{2}\big)\big(S_{2}vG_{2}vS_{2}\big)^{-1}\\
	&+\mu \big(S_{2}vG_{2}vS_{2}\big)^{-1}\Big[\big(S_{2}vG_{2}vS_{1}D_{1}S_{1}vG_{2}vS_{2}\big)\\
	&-\frac{a}{a_{2}^{3}}S_{2}vG_{3}vS_{2}\big(S_{2}vG_{2}vS_{2}\big)^{-1}S_{2}vG_{3}vS_{2}\Big]\big(S_{2}vG_{2}vS_{2}\big)^{-1}+O(\mu^2)
	\end{split}
	\end{equation*}
	
	From Lemma \ref{inverse formula}, we know
	\begin{equation}\label{M1}
	\begin{split}
	\big[M_{1}(\mu)\big]^{-1}=&\frac{1}{a}\Big(\frac{1}{a}M_{1}(\mu)+S_{2}\Big)^{-1}+\frac{1}{\mu a_{1}}\Big(\frac{1}{a}M_{1}(\mu)+S_{2}\Big)^{-1}S_{2}\big[M_{2}(\mu)\big]^{-1}S_{2}\Big(\frac{1}{a}M_{1}(\mu)+S_{2}\Big)^{-1}\\
	=&\frac{1}{\mu^{2}}\frac{1}{a}S_{2}\big(S_{2}vG_{2}vS_{2}\big)^{-1}S_{2}+\frac{1}{\mu}A_{-1}+A_{0}+O(\mu)
	\end{split}
	\end{equation}
	where $A_{-1}$ and $A_{0}$ are Hilbert-Schmidt operators which can be calculated clearly from the above expansions.
	Hence, we obtain the expansion \eqref{first}.
\end{proof}
	
\begin{lemma}
	Assume $|v(x)|\lesssim (1+|x|)^{-\beta/2}$ with some $\beta>11$. If zero is  an eigenvalue of $H$,  then for $<|\mu|\ll1$, we have
	\begin{equation}\label{third}
	\begin{split}
	\big[M(\mu)\big]^{-1}=&\mu^{-4}S_{3}\big(S_{3}vG_{3}vS_{3}\big)^{-1}S_{3}+\mu^{-3}B^{3}_{-3}+\mu^{-2}B^{3}_{ -2}+\mu^{-1}B^{3}_{ -1}+B^{3}_{0 }+O(\mu)
	\end{split}
	\end{equation}
	where $B^{3}_{ -3}, \,\,B^{3}_{-2},\,\, B^{3}_{-1},\,\, B^{3}_{0}$ are Hilbert-Schmidt operators.
\end{lemma}

\begin{proof}
	In the third resonance case, we aim to derive the asymptotic expansion of $\big[M_{2}(\mu)\big]^{-1}$. Denote $\widetilde{M}_{2}(\mu)=S_{2}vG_{2}vS_{2}+\mathscr{M}_{2}(\mu),$ then $$\big[M_{2}(\mu)\big]^{-1}=\frac{1}{\mu}\frac{a}{a_{2}}\big[\widetilde{M}_{2}(\mu)\big]^{-1}.$$ Since $S_{2}$ is of finite rank, thus 0 is an isolated spectrum point of self-adjoint operator $S_{2}vG_{3}vS_{2}$.  Let $S_{3}$ be the projcetion onto $\ker(S_{2}vG_{3}vS_{2})$, then $S_{2}vG_{3}vS_{2}+S_{3}$ is invertible on $S_2L^{2}$. Apply Lemma \ref{inverse formula} to $\widetilde{M}_{2}(\mu)$, since
	\begin{equation*}
	\begin{split}
	M_{3}(\mu)=&\frac{1}{\mu}\bigg[S_{3}-S_{3}\big(\widetilde{M}_{2}(\mu)+S_{3}\big)^{-1}S_{3}\bigg]\\
	=&\frac{1}{\mu}S_{3}\bigg[\mathscr{M}_{2}(\mu)D_{2}-\big(\mathscr{M}_{2}(\mu)D_{2}\big)+O(\mu^3)\bigg]S_{3}\\
	=&\frac{1}{a_{2}}S_{3}vG_{3}vS_{3}+\mu \bigg(-\frac{a_{2}}{a}S_{3}vG_{2}vS_{1}D_{1}S_{1}vG_{2}vS_{3}-\frac{1}{a_{2}^{2}}S_{3}vG_{3}vS_{2}D_{2}S_{2}vG_{3}vS_{3}\bigg)+O(\mu^2)\\
	:=&\frac{1}{a_{2}}S_{3}vG_{3}vS_{3}+\mathscr{M}_{3}(\mu)
	\end{split}
	\end{equation*}
	thus $M_{3}(\mu)$ is invertible provided $S_{3}vG_{3}vS_{3}$ is invertible on $S_{3}L^{2}$.
	Furthermore,
	\begin{equation*}
	\begin{split}
	\big[\widetilde{M}_{2}(\mu)\big]^{-1}=&\Big(\widetilde{M}_{2}(\mu)+S_{3}\Big)^{-1}+\frac{1}{\mu}\Big(\widetilde{M}_{2}(\mu)+S_{3}\Big)^{-1}S_{3}\big[M_{3}(\mu)\big]^{-1}S_{3}\Big(\widetilde{M}_{2}(\mu)+S_{3}\Big)^{-1}\\
	=&\frac{a_{2}}{\mu}S_{3}\big(S_{3}vG_{3}vS_{3}\big)^{-1}S_{3}+A+O(\mu),
	\end{split}
	\end{equation*}
	where
	\begin{equation*}
	\begin{split}
	A=&D_{2}+S_{3}\big(S_{3}vG_{3}vS_{3}\big)^{-1}\bigg(\frac{a_{2}^{3}}{a_{1}}S_{3}vG_{2}vS_{1}D_{1}S_{1}vG_{2}vS_{3}\\
	&+S_{3}vG_{3}vS_{2}D_{2}S_{2}vG_{3}vS_{3}\bigg)\big(S_{3}vG_{3}vS_{3}\big)^{-1}S_{3}\\
	&+D_{2}S_{2}vG_{3}vS_{3}\big(S_{3}vG_{3}vS_{3}\big)^{-1}S_{3}+S_{3}\big(S_{3}vG_{3}vS_{3}\big)^{-1}S_{3}vG_{3}vS_{2}D_{2}.
	\end{split}
	\end{equation*}
	Hence, \eqref{third} holds by substitute the expansion of $\big[M_{j}(\mu)\big]^{-1}, \, j=1, 2$ back to \eqref{M} and \eqref{M1} step by step.
\end{proof}

Now, we show that the kernel of the operator $S_{3}vG_{3}vS_{3}$ on $S_{3}L^{2}$ is trivial which implies the iterated process stop here.

\begin{lemma}\label{stop}
	Let $|v(x)|\lesssim(1+|x|)^{-\beta/2}$ with some $\beta>11$, then
	$$\ker (S_{3}vG_{3}vS_{3})=\{0\}.$$
\end{lemma}
\begin{proof}
	Take $f\in \ker(S_{3}vG_{3}vS_{3})$, then $f\in S_{3}L^{2}$. Then from Remark \ref{properties-of-S-T}, one have
	\begin{equation*}
	0=\langle v, f\rangle=\langle G_{2}vf, vf\rangle=\langle G_{3}vf, vf \rangle.
	\end{equation*}
Recall the expansion \eqref{free-resolvent-5} of $R_{0}(\mu^4)$:
\begin{equation*}
R_{0}(\mu^4; x, y)=G_{0}+a\mu P+a_{2}\mu^{3} G_{2}+\mu^{4}G_{3}+O(\mu^5).
\end{equation*}	
Therefore
\begin{equation*}
\begin{split}
0=&\langle G_{3}vf, vf \rangle\\
  =&\lim_{\mu\rightarrow 0}\Bigg\langle\frac{R^{+}_{}(\mu^4; x, y)-\Big(G_{0}+a_{1}\mu P+a_{2}\mu^{3} G_{2}+E_{0}(\mu)\Big)}{\mu^{4}}vf,\,\, vf \Bigg\rangle\\
  =&\lim_{\mu\rightarrow 0}\Bigg\langle\frac{R_{0}(\mu^4; x, y)-G_{0}}{\mu^{4}}vf,\,\, vf \Bigg\rangle
  =\lim_{\mu\rightarrow 0}\frac{1}{\mu^4}\Bigg\langle \Big(\frac{1}{\xi^{4}-\mu^4}-\frac{1}{\xi^4}\Big) \widehat{vf},\,\, \widehat{vf} \Bigg\rangle\\
  =&\lim_{\mu\rightarrow 0}\int_{\mathbf{R}^{5}}\frac{1}{(\xi^{4}-\mu^4)\xi^4}|\widehat{vf}|^{2}{\rm d}\xi=\int_{\mathbf{R}^{5}}\frac{|\widehat{vf}|^2}{\xi^8}{\rm d}\xi.
\end{split}
\end{equation*}
Here, we use dominated convergence as $\mu\rightarrow 0$ with chosen $\mu$ such that $\rm{Re} (\mu^4)<0$.
Thus $vf=0$ since $vf\in L^{1}$.  Note that $f\in S_{3}L^{2}\subset S_{1}L^{2}$, so
\begin{equation*}
0=T_{0}f=(U+vG_{0}v)f=Uf+vG_{0}vf,
\end{equation*}
Thus we have $f=-UvG_{0}vf$.  Hence $f=0$.
\end{proof}

\vspace{0.2cm}

Next, we give the proof of resolvent expansion for $d=6$.

{\bf Proof of Proposition \ref{RV-expansions-6} ($d=6$).}
In 6-dimensional case, by \eqref{free-resolvent-6-re}, in $B(s, -s)$ with $s>7$ we have
\begin{equation}
R_{0}(\mu^{4}; x, y)
=G_{0}+\mu^{2}c_{1}I+\mu^{4}c(\mu)G_{2}+\mu^{4}G_{3}+O(\mu^{6-})
\end{equation}
where $c(\mu)=c_{2}\ln\mu+c_{3}$ (see \eqref{free-resolvent-6}).

From the resolvent identity \eqref{symmetric-resolvent-idnetity}, since $|v(x)|\lesssim (1+|x|)^{-\beta}$ with some $\beta>14$, then in $B(0, 0)$ we have
\begin{equation}
M(\mu)=U+vG_{0}v+\mu^{2}cP+\mu^4 c(\mu)vG_{2}v+\mu^{4}vG_{3}v+O(\mu^{6-})
\end{equation}
with $c=a_{1}\|v\|_{L^2(\mathbf{R}^{6})}$.

Proposition \ref{RV-expansions-6} holds by the following lemmas: Lemma \ref{regular-6} - Lemma \ref{third6}.
\begin{lemma}\label{regular-6}
	Assume $|v(x)|\lesssim(1+|x|)^{-\beta/2}$ with some $\beta>14$. If zero is a regular  point of $H$, then for $0<|\mu|\ll1$, in $B(0, 0)$ we have
	\begin{equation}
	\big[M(\mu)\big]^{-1}=T_{0}^{-1}-\mu^{2}T_{0}^{-1}cPT_{0}^{-1}-\mu^4c(\mu)T_{0}^{-1}vG_{2}vT_{0}^{-1}+O(\mu^4).
	\end{equation}
\end{lemma}

\begin{lemma}\label{first6}
	Assume $|v(x)|\lesssim(1+|x|)^{-\beta/2}$ with some $\beta>14$. If zero is of the first kind resonance of $H$, then for $0<|\mu|\ll1$ in $B(0, 0)$ we have
	\begin{equation}\label{first-6}
	\begin{split}
	\big[M(\mu)\big]^{-1}=&\mu^{-2}\frac{1}{c}S_{1}\big(S_{1}PS_{1}\big)^{-1}-\frac{c(\mu)}{c^{2}}S_{1}\big(S_{1}PS_{1}\big)^{-1}S_{1}vG_{2}vS_{1}\big(S_{1}PS_{1}\big)^{-1}S_{1}\\
	&+B_{0}+O(\mu^{2-}).
	\end{split}
	\end{equation}
\end{lemma}

\begin{proof}
	Let $S_{1}$ be the projection onto $\ker(U+vG_{0}v)$, then apply Lemma \ref{Feshbach-formula} to $M(\mu)$, we have
	\begin{equation*}
	\big[M(\mu)\big]^{-1}=\Big(M(\mu)+S_{1}\Big)^{-1}+\Big(M(\mu)+S_{1}\Big)^{-1}S_{1}\big[M_{1}(\mu)\big]^{-1}S_{1}\Big(M(\mu)+S_{1}\Big)^{-1}
	\end{equation*}
	where
	\begin{equation*}
	\begin{split}
	M_{1}(\mu)=&S_{1}-S_{1}\Big(M(\mu)+S_{1}\Big)^{-1}S_{1}\\
	=&S_{1}-S_{1}D_{0}\bigg[1+\Big(\mu^{2}P+\mu^4c(\mu)vG_{2}v+\mu^{4}vG_{3}v+O(\mu^{6-})\Big)D_{0}\bigg]^{-1}S_{1}\\
	=&\mu^{2}cS_{1}PS_{1}+\mu^{4}c(\mu)S_{1}vG_{2}vS_{1}+\mu^{4}S_{1}vG_{3}vS_{1}-\mu^{4}c^{2}S_{1}PD_{0}PS_{1}+O(\mu^{6-}).
	\end{split}
	\end{equation*}
	Thus $\mu^{2}M_{1}(\mu)$ can be expanded into Neumann series provided $S_{1}PS_{1}$ is invertible.  Hence, \eqref{first-6} holds by expand $\Big(M(\mu)+S_{1}\Big)^{-1}$ and $\mu^{2}M_{1}(\mu)$ into Neumann series.
\end{proof}

\begin{lemma}\label{second6}
	Assume $|v(x)|\lesssim(1+|x|)^{-\beta/2}$ with some $\beta>14$. If zero is of the second kind resonance of $H$, then for  $0<|\mu|\ll1$ in $B(0, 0)$ we have
	\begin{equation}\label{second-6}
	\begin{split}
	\big[M(\mu)\big]^{-1}=&\frac{1}{\mu^4c(\mu)}S_{2}\big(S_{2}vG_{2}vS_{2}\big)^{-1}S_{2}+\frac{1}{c\mu^2}\bigg(S_{1}D_{1}S_{1}-D_{1}S_{1}vG_{2}vS_{2}\\
	&\cdot\big(S_{2}vG_{2}vS_{2}\big)^{-1}S_{2}-S_{2}\big(S_{2}vG_{2}vS_{2}\big)^{-1}S_{2}vG_{2}vS_{1}D_{1}\bigg)\\
	&-\frac{c(\mu)}{c^{2}}D_{1}S_{2}vG_{2}vS_{2}D_{1}+O\big((\ln\mu)^{-1}\big).
	\end{split}
	\end{equation}
\end{lemma}

\begin{proof}
	In the second resonance case, we aim to derive the expansion of $\big[M_{1}(\mu)\big]^{-1}$. Denote
	\begin{equation*}
	\mathscr{M}_{1}(\mu)=S_{1}PS_{1}+\frac{\mu^2c(\mu)}{c}S_{1}vG_{2}vS_{1}+\frac{\mu^2}{c}S_{1}vG_{3}vS_{1}-\frac{\mu^2}{c}S_{1}PD_{0}PS_{1}+O(\mu^{4-}),
	\end{equation*}
	then $\big[M(\mu)\big]^{-1}=\frac{1}{c\mu^2}\big[\mathscr{M}_{1}(\mu)\big]^{-1}$.
	
	Since $S_{2}$ is the projection onto $\ker(S_{1}PS_{1})$, then $PS_{2}=S_{2}P=0$.  Apply Lemma \ref{Feshbach-formula} to $\mathscr{M}_{1}(\mu)$, then
	\begin{equation*}
	\big[\mathscr{M}_{1}(\mu)\big]^{-1}=\big(\mathscr{M}_{1}(\mu)+S_{2}\big)^{-1}+\big(\mathscr{M}_{1}(\mu)+S_{2}\big)^{-1}S_{2}\big[M_{2}(\mu)\big]^{-1}S_{2}\big(\mathscr{M}_{1}(\mu)+S_{2}\big)^{-1}
	\end{equation*}
	where
	\begin{equation*}
	\begin{split}
	{M}_{2}(\mu)=&S_{2}-S_{2}\big(\mathscr{M}_{1}(\mu)+S_{2}\big)^{-1}S_{2}\\
	=&S_{2}-S_{2}\bigg[1+\Big(\frac{\mu^2c(\mu)}{c}S_{1}vG_{2}vS_{1}+\frac{\mu^2}{c}S_{1}vG_{3}vS_{1}-\frac{\mu^2}{a}S_{1}PD_{0}PS_{1}+O(\mu^{4-})\Big)D_{1}\bigg]^{-1}S_{2}\\
	=&\frac{\mu^2c(\mu)}{c}S_{2}vG_{2}vS_{2}+\frac{\mu^2}{a}S_{2}vG_{3}vS_{2}+O(\mu^{4-}).
	\end{split}
	\end{equation*}
	Thus $\frac{c}{\mu^2c(\mu)}M_{2}(\mu)$ is invertible and can be expanded into Neumann series.
\end{proof}

\begin{lemma}\label{third6}
	Assume $|v(x)|\lesssim(1+|x|)^{-\beta}$ with some $\beta>14$. If zero is an eigenvalue of $H$, then for $0<|\mu|\ll1$ in $B(0, 0)$ we have
	\begin{equation}\label{third-6}
	\begin{split}
	\big[M(\mu)\big]^{-1}=&\frac{1}{\mu^4}S_{3}\big(S_{3}vG_{3}vS_{3}\big)^{-1}S_{3}+\frac{B_{3,-4}}{g(\mu)}+\frac{B_{3,-3}}{\mu^2}\\
	&+(\ln\mu)^2B_{3,-2}+(\ln\mu)B_{3,-1}+B_{3,0}+O\big((\ln\mu)^{-1}\big).
	\end{split}
	\end{equation}
	where  $B_{3,-4}, B_{3,-2}, B_{3,-2}, B_{3,-1}, B_{3,0}$ are $\mu$ independent Hilbert-Schmidt operators.
\end{lemma}

\begin{proof}
	In the third kind ``resonance" case, we aim to derive the expansion of $\big[M_2(\mu)\big]^{-1}$.  Let $\mathscr{M}_{2}(\mu)=\frac{c}{\mu^2c(\mu)}M_{2}(\mu)$.  Applying Lemma \ref{Feshbach-formula} to $\mathscr{M}_{2}(\mu)$, then
	\begin{equation*}
	\big[\mathscr{M}_{2}(\mu)\big]^{-1}=\big(\mathscr{M}_{2}(\mu)+S_{3}\big)^{-1}+\big(\mathscr{M}_{2}(\mu)+S_{3}\big)^{-1}S_{3}\big[M_{3}(\mu)\big]^{-1}S_{3}\big(\mathscr{M}_{2}(\mu)+S_{3}\big)^{-1}
	\end{equation*}
	where
	\begin{equation*}
	\begin{split}
	{M}_{3}(\mu)=&S_{3}-S_{3}\big(\mathscr{M}_{2}(\mu)+S_{3}\big)^{-1}S_{3}\\
	=&S_{3}-S_{3}\bigg[1+\Big(\frac{1}{c(\mu)}S_{2}vG_{3}vS_{2}+\frac{O(\mu^{2-})}{c(\mu)}\Big)D_{2}\bigg]^{-1}S_{3}\\
	=&\frac{1}{c(\mu)}S_{3}vG_{3}vS_{3}+\frac{O(\mu^{2-})}{c(\mu)}.
	\end{split}
	\end{equation*}
	By the same discussion as in Lemma \ref{stop}, we know $\ker(S_{3}vG_{3}vS_{3})=\{0\}$,  hence $M_{3}(\mu)$ is invertible and \eqref{third-6} holds by expanding $\big[M_{3}(\mu)\big]^{-1}$ into Neumann series.
\end{proof}

{\bf Proof of Proposition \ref{RV-expansions-d} ($d\geq9$).} In the cases $d\geq 9$, by Lemma \ref{free-expansions}, in $B(s, -s)$ with $s>d/2$, we have
\begin{equation}\label{free-resolvent-9-G}
R_{0}(\mu^4; x, y)=G_{0}+\mu^{4}G_{1}+\mu^{5}G(\mu)+\mu^{d-4} c_{0}(d)G_{2}+O(\mu^{(d-4)+}).
\end{equation}

Since $|v(x)|\lesssim (1+|x|)^{-\beta/2}$ with some $\beta>d$, then in $B(0, 0)$ we have
\begin{equation}
M(\mu)=U+vG_{0}v+\mu^{4}vG_{1}v+\mu^{5}vG(\mu)v+c_{0}(d)\mu^{d-4} P+O(\mu^{d-4+}).
\end{equation}

Proposition \ref{RV-expansions-d} holds by the following Lemma \ref{regular-9} and the symmetry resolvent identity \eqref{symmetric-resolvent-idnetity}.

\begin{lemma}\label{regular-9}
	Assume $|v(x)|\lesssim (1+|x|)^{-\beta/2}$ with some $\beta>d$.  If zero is a regular point  of $H$, then for $0<|\mu|\ll 1$,
	\begin{equation}\label{regular-expansion-9}
\big[M(\mu)\big]^{-1}=T_{0}^{-1}-\mu^{4} T_{0}^{-1}vG_{1}vT_{0}^{-1}-\mu^{5}T_{0}^{-1}vG(\mu)vT_{0}^{-1}-c_{0}(d)\mu^{d-4}T_{0}^{-1}PT_{0}^{-1}+O(\mu^{d-4+}).
	\end{equation}
	
	If zero is an eigenvalue of $H$, then for $0<|\mu|\ll 1$,
	\begin{equation}\label{first-9}
\big[M(\mu)\big]^{-1}=\mu^{-4}S_{1}\big(S_{1}vG_{1}vS_{1}\big)^{-1}S_{1}-\mu^{-3}S_{1}vG(\mu)vS_{1}-c_{0}(d)\mu^{d-12} B_{-1}^{d}+B_{0}^{d}+O(\mu^{d-8}),
	\end{equation}
	where
	\begin{align*}
	&B_{-1}^{d}=S_{1}\big(S_{1}vG_{1}vS_{1}\big)^{-1}P\big(S_{1}vG_{1}vS_{1}\big)^{-1}S_{1};\\
	&B_{0}^{d}=-D_{0}vG_{1}vS_{1}\big(S_{1}vG_{1}vS_{1}\big)^{-1}S_{1}-S_{1}\big(S_{1}vG_{1}vS_{1}\big)^{-1}S_{1}vG_{1}vD_{0}.
	\end{align*}
\end{lemma}

\begin{proof}
	If zero is a regular point, then \eqref{regular-9} holds by expanding $\big[M(\mu)\big]^{-1}$ into Neumann series. If zero is an eigenvalue, then by Lemma \ref{inverse formula}, we have
	\begin{equation*}
	\begin{split}
	\big[M(\mu)\big]^{-1}=&\big(M(\mu)+S_{1}\big)^{-1}+\frac{1}{\mu^4}\big(M(\mu)+S_{1}\big)^{-1}S_{1}\big[M_1(\mu)\big]^{-1}S_{1}\big(M(\mu)+S_{1}\big)^{-1},
	\end{split}
	\end{equation*}
	where
	\begin{equation*}
	\begin{split}
	M_{1}(\mu)=&\frac{1}{\mu^4}\bigg[S_{1}-S_{1}\big(M(\mu)+S_{1}\big)^{-1}S_{1}\bigg]\\
	=&S_{1}vG_{1}vS_{1}+\mu S_{1}vG(\mu)vS_{1}+c_{0}(d)\mu^{d-8}S_{1}PS_{1}+O(\mu^{d-8+}).
	\end{split}
	\end{equation*}
	
	By the same strategy of Proposition \ref{stop}, one can prove that $\ker(S_{1}vG_{1}vS_{1})=\{0\}$. Hence, \eqref{first-9} holds by expanding $\big[M_1(\mu)\big]^{-1}$ and $\big(M(\mu)+S_{1}\big)^{-1}$ into Neumann series.
\end{proof}

\subsection{Identification of resonance spaces}  In this part, we aim to identify the resonance subspace of each kind for different dimensional cases with $d\geq5$.

{\bf Proof of Proposition \ref{classification-5}($d=5$).}  The proof  is divided into the following three lemmas. (See Lemma \ref{S1} - \ref{S3})
\begin{lemma}\label{S1}
	Let $|v(x)|\lesssim(1+|x|)^{-\beta}$ with some $\beta>11$.  $\phi\in S_{1}L^{2}\setminus\{0\}$ if and only if $\phi=Uv\psi$ where $\psi\in W_{3/2}(\mathbf{R}^{5})$  satisfies $H\psi=0$ in distributional sense, and $\psi=-G_{0}v\phi.$
\end{lemma}

\begin{proof}
	Since $\phi\in S_{1}L^{2}$, then
	\begin{equation*}
	0=(U+vG_{0}v)\phi=U\phi+vG_{0}v\phi,
	\end{equation*}
	so that $\phi=Uv(-G_{0}v\phi)$.
	
	Since $(-\Delta)^2G_{0}v\phi=v\phi$ distributional, then
	\begin{equation*}
	[(-\Delta)^2+V](-G_{0}v\phi)=-v\phi+V(-G_{0}v\phi)=-vUv\psi+V\psi=0.
	\end{equation*}
	Now, we show $\psi=-G_{0}v\in W_{3/2}(\mathbf{R}^{5})$ .  Note that
	\begin{equation*}
	\begin{split}
	G_{0}v\phi=&c\int_{\mathbf{R}^{5}}|x-y|^{-1}v(y)\phi(y){\rm d}y=I_{4}v(y)\phi(y).
	\end{split}
	\end{equation*}
	Hence, by Lemma \ref{Reisz-potential-boundedness} and the fact that $\psi=-G_{0}V\psi$, then we know $\psi\in W_{3/2}(\mathbf{R}^{5})$.
\end{proof}

\begin{lemma}\label{S2}
	Let $|v(x)|\lesssim(1+|x|)^{-\beta}$ with some $\beta>11$. Then $\phi=Uv\psi\in S_{2}L^{2}$ if and only if $\psi\in W_{1/2}(\mathbf{R}^{5})$.
\end{lemma}
\begin{proof}
	First we show that if $\phi\in S_{2}L^2$ then $\psi\in W_{1/2}(\mathbf{R}^{5})$. Since $\phi\in S_{2}L^{2}$, then $S_{1}P\phi=0$ which implies that $P\phi=0$. Indeed, $0=\langle S_{1}P\phi, \phi\rangle=\langle P\phi, S_{1}\phi\rangle=\langle P\phi, P\phi\rangle$.
	
	Using $S_{2}\leq S_{1}$, then by Lemma \ref{S1}, we have
	\begin{equation*}
	\begin{split}
	G_{0}v\psi=\int_{\mathbf{R}^{5}}\bigg[\frac{1}{|x-y|}-\frac{1}{1+|x|}\bigg]v(y)\phi(y){\rm d}y
	\leq\frac{1}{1+|x|}\int_{\mathbf{R}^{5}}\frac{1+|y|}{|x-y|}\Big|v(y)\phi(y)\Big|{\rm d}y.
	\end{split}
	\end{equation*}
	Thus by the boundedness of $I_{4}$ in $\mathbf{R}^{5}$ and the fact that $\psi=-G_{0}V\psi$, we know $\psi\in W_{1/2}(\mathbf{R}^{5})$.
	
	Next, we assume $\psi\in W_{1/2}(\mathbf{R}^{5})$ and $\phi=Uv\psi$ and show
	\begin{equation}\label{S_{2}}
	\int_{\mathbf{R}^{5}}v(y)\phi(y){\rm d}y=0.
	\end{equation}
	Since
	\begin{equation*}
	\begin{split}
	\psi=\int_{\mathbf{R}^{5}}\bigg[\frac{1}{|x-y|}-\frac{1}{1+|x|}\bigg]v(y)\phi(y){\rm d}y+\int_{\mathbf{R}^{5}}\frac{1}{1+|x|}v(y)\phi(y){\rm d}y.
	\end{split}
	\end{equation*}
	If $\psi\in W_{1/2}(\mathbf{R}^{5})$, then
	\begin{equation*}
	\frac{1}{1+|x|}\int_{\mathbf{R}^{5}}v(y)\phi(y){\rm d}y\in W_{1/2}(\mathbf{R}^{5}).
	\end{equation*}
	Hence \eqref{S_{2}} holds and $\phi\in S_{2}L^{2}$.
\end{proof}

\begin{lemma}\label{S3}
	Let $|v(x)|\lesssim(1+|x|)^{-\beta}$ with some $\beta>11$. Then $\phi=Uv\psi\in S_{3}L^{2}$ if and only if $\psi\in L^{2}(\mathbf{R}^{5})$.
\end{lemma}
\begin{proof}
	Since $\phi\in S_{3}L^{2}$, then $S_{2}vG_{2}v\phi=0$. Thus using $S_{3}\leq S_{2}\leq S_{1}$ then
	\begin{equation*}
	\begin{split}
	\langle v\phi, G_{2}v\phi\rangle=&\int_{\mathbf{R}^{5}}\int_{\mathbf{R}^{5}}\phi(x)v(x)|x-y|^{2}v(y)\phi(y){\rm d}x{\rm d}y\\
	=&\int_{\mathbf{R}^{5}}\int_{\mathbf{R}^{5}}\phi(x)v(x)\Big[|x|^{2}+|y|^{2}-2x\cdot y\Big]v(y)\phi(y){\rm d}x{\rm d}y\\
	=&\int_{\mathbf{R}^{5}}\int_{\mathbf{R}^{5}}\phi(x)v(x)\Big[-2x\cdot y\Big]v(y)\phi(y){\rm d}x{\rm d}y\\
	=&-2\Bigg|\int_{\mathbf{R}^{5}}yv(y)\phi(y){\rm d}y\Bigg|^2=0.
	\end{split}
	\end{equation*}
	Hence, for $1\leq j\leq 5$,
	\begin{equation}\label{G_2}
	\int_{\mathbf{R}^{5}}y_{j}v(y)\phi(y){\rm d}y=0.
	\end{equation}
	
	Next, we show $\psi\in L^{2}(\mathbf{R}^{5})$. Since
	\begin{equation*}
	\begin{split}
	\psi=&\int_{\mathbf{R}^{5}}\Big[\frac{1}{|x-y|}-\frac{1}{1+|x|}-\frac{1}{(1+|x|)^{2}}-\frac{2x\cdot y}{(1+|x|)^{3}}\Big]v(y)\phi(y){\rm d}y\\
	=&\int_{\mathbf{R}^{5}}\Big[\frac{1}{|x-y|(1+|x|)}-\frac{1}{(1+|x|)^{2}}\Big]v(y)\phi(y){\rm d}y\\
	&+\int_{\mathbf{R}^{5}}\Big[\frac{|x|-|x-y|}{(1+|x|)|x-y|}-\frac{2x\cdot y}{(1+|x|)^{3}}\Big]v(y)\phi(y){\rm d}y\\
	:=&I+II,
	\end{split}
	\end{equation*}
	then we show $I, II\in L^{2}(\mathbf{R}^{5})$.
	
	For term $I$, from \eqref{G_2}, we have
	\begin{equation*}
	\begin{split}
	I=\int_{\mathbf{R}^{5}}\frac{1+|x|-|x-y|}{|x-y|(1+|x|)^{2}}v(y)\phi(y){\rm d}y
    \leq\int_{\mathbf{R}^{5}}\Bigg[\frac{1+|y|}{|x-y|(1+|x|)^{2}}\Bigg]\big|v(y)\phi(y)\Big|{\rm d}y\in L^{2}(\mathbf{R}^5).
	\end{split}
	\end{equation*}
	For term $II$, from \eqref{S_{2}}, we have
	\begin{equation*}
	\begin{split}
	II=&\int_{\mathbf{R}^{5}}\Bigg[\frac{2x\cdot y-|y|^{2}}{(1+|x|)|x-y|(|x-y|+|x|)}-\frac{2x\cdot y}{(1+|x|)^{3}}\Bigg]v(y)\phi(y){\rm d}y\\
	=&\int_{\mathbf{R}^{5}}\Bigg[-\frac{|y|^{2}}{(|x-y|+|x|)|x-y|(1+|x|)}\Bigg]v(y)\phi(y){\rm d}y\\
	&-\int_{\mathbf{R}^{3}}2x\cdot y\Bigg[\frac{(1+|x|)^{2}-|x-y|(|x|+|x-y|)}{(|x-y|+|x|)|x-y|(1+|x|)^{3}}\Bigg]v(y)\phi(y){\rm d}y\\
	\leq&\int_{\mathbf{R}^{5}}\Bigg[\frac{|y|^{2}}{|x-y|^{2}(1+|x|)}\Bigg]\Big|v(y)\phi(y)\Big|{\rm d}y\\
	&+\int_{\mathbf{R}^{5}}2|x||y|\Bigg[\frac{1+2|x|(1+|y|)+|y|^2+|x-y||x|}{|x-y|^{2}(1+|x|)^{3}}\Bigg]v(y)\phi(y){\rm d}y\\
	\leq&\int_{\mathbf{R}^{5}}\Bigg[\frac{|y|^{2}}{|x-y|^{2}(1+|x|)}\Bigg]\Big|v(y)\phi(y)\Big|{\rm d}y+\int_{\mathbf{R}^{5}}\Bigg[\frac{2(1+|y|)}{|x-y|^{2}(1+|x|)^{2}}+\frac{4(1+|y|)}{|x-y|^2(1+|x|)}\\
	&+\frac{2(1+|y|)^{5}}{|x-y|^{2}(1+|x|)}+\frac{2(1+|y|)}{|x-y|(1+|x|)^2}\Bigg]\Big|v(y)\phi(y)\Big|{\rm d}y.
	\end{split}
	\end{equation*}
	Hence, by Lemma \ref{Reisz-potential-boundedness}, we know $II\in L^{2}(\mathbf{R}^{5})$.
	
	Now, we show if $\psi\in L^{2}(\mathbf{R}^{5})$, then $\phi\in S_{3}L^{2}$.  Since  $\psi\in L^{2}(\mathbf{R}^{5})\subset L^{2}_{ -1/2-}(\mathbf{R}^5)$, by Lemma \ref{S2}, then $\phi\in S_{2}L^{2}$,  thus $\phi$ satisfies
	\begin{equation*}
	\int_{\mathbf{R}^{5}}v(y)\phi(y)dy=0.
	\end{equation*}
	Hence, for $\psi\in L^{2}(\mathbf{R}^5)$ we have
	\begin{equation*}
	\begin{split}
	\psi=&\int_{\mathbf{R}^{5}}\Big[\frac{1}{|x-y|}-\frac{1}{1+|x|}-\frac{1}{(1+|x|)^{2}}-\frac{2x\cdot y}{(1+|x|)^{3}}\Big]v(y)\phi(y){\rm d}y\\
	=&\int_{\mathbf{R}^{5}}\Big[\frac{1}{|x-y|(1+|x|)}-\frac{1}{(1+|x|)^{2}}\Big]v(y)\phi(y){\rm d}y+\int_{\mathbf{R}^{5}}\Big[\frac{|x|-|x-y|}{(1+|x|)|x-y|}\Big]v(y)\phi(y){\rm d}y\\
	\leq &\int_{\mathbf{R}^{5}}\Bigg[\frac{1+|y|}{|x-y|(1+|x|)^{2}}\Bigg]\Big|v(y)\phi(y)\Big|{\rm d}y+\int_{\mathbf{R}^{5}}\Bigg[\frac{|y|^2}{|x-y|(1+|x|)(|x|+|x-y|)}\Bigg]\Big|v(y)\phi(y)\Big|{\rm d}y\\
	&+\int_{\mathbf{R}^{5}}\Bigg[\frac{2x\cdot y}{|x-y|(1+|x|)(|x|+|x-y|)}\Bigg]v(y)\phi(y){\rm d}y\\
	=&\int_{\mathbf{R}^{5}}\Bigg[\frac{2x\cdot y}{|x-y|(1+|x|)(|x|+|x-y|)}\Bigg]v(y)\phi(y){\rm d}y+O_{L^{2}}(1)
	\end{split}
	\end{equation*}
	For $|x|\gg1$,  we have
	\begin{equation*}
	\begin{split}
	&\int_{\mathbf{R}^{5}}\Bigg[\frac{2x\cdot y}{|x-y|(1+|x|)(|x|+|x-y|)}\Bigg]v(y)\phi(y){\rm d}y\\
	\simeq&(1+|x|)^{-3}\int_{\mathbf{R}^{5}}\sum_{1\leq j\leq 5}x_{j}y_{j}v(y)\phi(y){\rm d}y\in L^{2}(|x|\gg1).
	\end{split}
	\end{equation*}
	Hence, for $1\leq j\leq 5$
	\begin{equation*}
	\int_{\mathbf{R}^{5}}y_{j}v(y)\phi(y){\rm d}y=0
	\end{equation*}
	which implies that $\phi\in S_{3}L^{2}$.
\end{proof}

{\bf Proof of Theorem \ref{classification-6} ($d=6$).} The proof of $d=6$ is divided into the following three lemmas, see Lemma \ref{s1-space} - \ref{s3-space}.
\begin{lemma}\label{s1-space}
	Suppose $|v(x)|\lesssim (1+ |x|)^{-\beta}$ with some $\beta>14$. Then $\phi\in S_1L^2\setminus\{0\}$ if and only if $\phi=Uv\psi$ for some $\psi\in W_{1}(\mathbf{R}^{6})\setminus\{0\}$  such that
	$[(-\Delta)^2+V]\psi=0$
	holds in the sense of distributions.
\end{lemma}

\begin{proof}
	We first note that
	\[[(-\Delta)^2+V]\psi=0 \Leftrightarrow (I+G_0V)\psi=0.\]
	Firstly, if $\phi\in S_1L^2\setminus\{0\}.$ Then $(U+vG_0v)\phi=0,$ and multiplying by $U$, one has
	\[\phi(x)=-UvG_0v\phi=-c_0Uv(x)\int_{\mathbf{R}^6}\frac{v(y)\phi(y)}{|x-y|^2}{\rm d}y.\]
	Accordingly, we define
	\[\psi(x)=-G_0v\phi=c_1\int_{\mathbf{R}^6}\frac{v(y)\phi(y)}{|x-y|^2}{\rm d}y.\]
	Since $v\phi\in L^{2}_{\beta/2}(\mathbf{R}^6)$,  by Lemma \ref{Reisz-potential-boundedness} and the fact that $\psi=-G_{0}V\psi$, we have that $\psi\in W_{1}(\mathbf{R}^6)$. Further $\phi(x)=Uv(x)\psi(x)$ and
	\[\psi(x)=-G_0v\phi=-G_0V\psi(x) \Rightarrow (I+G_0V)\psi=0.\]
	
	Secondly, assume $\phi=Uv\psi$ for a non-zero distributional solution to $[(-\Delta)^2+V]\psi=0.$ It is clear that $\phi\in L^{2}(\mathbf{R}^6)$ and
	\[(U+vG_0v)\phi(x)=v(x)\phi(x)+v(x)G_0V\psi(x)=v(x)(I+G_0V)\psi(x)=0.\]
	Hence, $\phi\in S_1L^2\setminus\{0\}.$
\end{proof}

\begin{lemma}\label{s2-space}
	Suppose $|v(x)|\lesssim (1+|x|)^{-\beta}$ with some $\beta>14$. Then $\phi=Uv\psi\in S_2L^2\setminus\{0\}$ if and only if $\psi\in W_{0}(\mathbf{R}^{6})\setminus\{0\}$ .
\end{lemma}
\begin{proof}
	Assume first that $\psi\in S_2L^2\setminus\{0\},$ since $S_2\le S_1$. Using Lemma \ref{s1-space}, we need only to show that $\psi\in W_{0}(\mathbf{R}^{6}).$ Since $Pf=0,$ we have
	\[\int_{\mathbf{R}^6}v(y)f(y){\rm d}y=0.\]
	Then
	\[\psi(x)=-c_1\int_{\mathbf{R}^6}\Big[\frac{1}{|x-y|^2}-\frac{1}{1+|x|^2}\Big]v(y)f(y){\rm d}y.\]
	Using
	\[\Big|\frac{1}{|x-y|^2}-\frac{1}{1+|x|^2}\Big|
	\lesssim\frac{1+|y|}{(1+|x|)|x-y|^2}+\frac{1+|y|}{|x-y|(1+|x|)^2}\]
	and noting that $(1+|\cdot|)vf\in L^{2}_{4+}(\mathbf{R}^{6})$, the Riesz potential $I_4$ maps $L^{2}_{4+}(\mathbf{R}^{6})$ to $L^{2}_{-1-}(\mathbf{R}^{6})$ and $I_5$ maps $L^{2}_{4+}(\mathbf{R}^{6})$ to $L^{2}_{-2-}(\mathbf{R}^{6})$
	shows that $\psi\in W_{0}(\mathbf{R}^{6})$ as desired.\\
	\indent	On the other hand, if $\phi=Uv\psi$ as in hypothesis, we have
	\begin{equation}\label{eq-s2-1}
	\psi(x)=-c_1\int_{\mathbf{R}^6}\Big[\frac{1}{|x-y|^2}-\frac{1}{1+|x|^2}\Big]v(y)f(y){\rm d}y-\frac{c_1}{1+|x|^2}\int_{\mathbf{R}^6}v(y)f(y){\rm d}y.
	\end{equation}
	The first term and $\psi(x)$ are in $W_{0}(\mathbf{R}^{6})$,  thus we must have that
	\[-\frac{c_1}{1+|x|^2}\int_{\mathbf{R}^6}v(y)f(y){\rm d}y\in L^{2}_{0-}(\mathbf{R}^6),\]
	then we know that
	\[\int_{\mathbf{R}^6}v(y)f(y)dy=0,\]
	that is $0=Pf=S_1PS_1f$ and $f\in S_2L^2$ as desired.
\end{proof}

\begin{lemma}\label{s3-space}
	Suppose $v(x)\lesssim (1+|x|)^{-\beta}$ with some $\beta>14$. Then $\phi=Uv\psi\in S_3L^2\setminus\{0\}$ if and only $\psi \in L^{2}(\mathbf{R})\setminus\{0\}$.	
\end{lemma}
\begin{proof}
	Assume first that $\phi\in S_3L^2\setminus\{0\}$.  Since $S_3\le S_2\le S_1$,  from Lemma \ref{s1-space} and Lemma \ref{s2-space}, we need only to show that $\psi\in L^2(\mathbf{R}^{6})$. By the definition of $S_3,$ we know that $\phi$ is in the kernel of $S_2vG_3vS_2,$ that is $S_2vG_3vf=0,$ then we have
	\begin{align*}
	0&=\langle S_2vG_3vf, f\rangle=\langle G_3vf, vf\rangle\\
	&=\int_{\mathbf{R}^6}dx\int_{\mathbf{R}^6}v(x)f(x)[|x|^2-2x\cdot  y+|y|^2]v(y)f(y){\rm d}y
	\end{align*}
	Since $Pf=0,$ that is
	\[\int_{\mathbf{R}^6}v(y)f(y){\rm d}y=0\]
	hence we can obtain that
	\[-2\int_{\mathbf{R}^6}v(x)f(x)dx\int_{\mathbf{R}^6}x\cdot yv(y)f(y){\rm d}y=-2\Bigg[\int_{\mathbf{R}^6}yv(y)f(y){\rm d}y\Bigg]^2=0\]
	which gives
	\[\int_{\mathbf{R}^6}yv(y)f(y){\rm d}y=0.\]
	Then,
	\[\psi(x)=-c_1\int_{\mathbf{R}^6}\Big[\frac{1}{|x-y|^2}-\frac{1}{1+|x|^2}-\frac{x\cdot y}{(1+|x|^2)^2}
	\Big]v(y)f(y){\rm d}y.\]
	Since
	\[\Big|\frac{1}{|x-y|^2}-\frac{1}{(1+|x|)^2}-\frac{x\cdot y}{(1+|x|)^2}\Big|\lesssim\frac{(1+|y|)^2}{|x-y|^2(1+|x|)^2}+\frac{(1+|y|)^3}{|x-y|^2(1+|x|)^3}\]
	and noting that $(1+|\cdot|)^3vf\in L^{2}_{2+}(\mathbf{R}^{6}),\,  (1+|\cdot|)^2vf\in L^{2}_{3+}(\mathbf{R}^{6}),$ Lemma \ref{Reisz-potential-boundedness} shows that $\psi\in L^2(\mathbf{R}^{6})$ as desired.\\
	\indent On the other hand, if $\phi=Uv\psi$ as in hypothesis,according to the proof in Lemma \ref{s2-space}, we can similarly prove that
	\[\int_{\mathbf{R}^6}v(y)f(y){\rm d}y=0,\ \ \int_{\mathbf{R}^6}yv(y)f(y){\rm d}y=0\]
	which means that $S_1PS_1f=0$ and $S_2vG_3vS_2f=0,$ that is $f\in S_3L^2.$
\end{proof}

\begin{proof}[{\bf Proof of Theorem \ref{classification-7} ($d=7$).} ]
(1) Since $\phi\in S_{1}L^{2}$, then
\begin{equation*}
0=(U+vG_{0}v)\phi=U\phi+vG_{0}v\phi,
\end{equation*}
so that $\phi=Uv(-G_{0}v\phi)$.

Since $(-\Delta)^2G_{0}v\phi=v\phi$ distributional, then
\begin{equation*}
[(-\Delta)^2+V](-G_{0}v\phi)=-v\phi+V(-G_{0}v\phi)=-vUv\psi+V\psi=0.
\end{equation*}
Now, we show $\psi=-G_{0}v\phi\in W_{1/2}(\mathbf{R}^{7})$ . Since
\begin{equation*}
\begin{split}
G_{0}v\phi=&c\int_{\mathbf{R}^{7}}|x-y|^{-3}v(y)\phi(y){\rm d}y=I_{4}v(y)\phi(y).
\end{split}
\end{equation*}
Hence, by Lemma \ref{Reisz-potential-boundedness} and  $\psi=-G_{0}V\psi$, we know $-G_{0}v\phi\in W_{1/2}(\mathbf{R}^{7})$.

Notice that, since $\frac{1}{|x-y|^3}\sim \frac{1}{|x|^3}$ for $|x|\gg1$, thus the rate of the weight $-1/2-$ is sharp.

(2) First we show that if $\phi\in S_{2}L^2$ then $\psi\in L^{2}(\mathbf{R}^{7})$. Since $\phi\in S_{2}L^{2}$, then $S_{1}P\phi=0$ which implies that $P\phi=0$. Indeed, $0=\langle S_{1}P\phi, \phi\rangle=\langle P\phi, S_{1}\phi\rangle=\langle P\phi, P\phi\rangle$.

Using $S_{2}\leq S_{1}$ and $P\phi=0$, then
\begin{equation*}
\begin{split}
G_{0}v\psi=&\int_{\mathbf{R}^{7}}\bigg[\frac{1}{|x-y|^3}-\frac{1}{(1+|x|)^{3}}\bigg]v(y)\phi(y){\rm d}y\\
\lesssim&\int_{\mathbf{R}^{7}}\frac{(1+|y|)^{2}(1+|x|)+(1+|y|)^{3}}{|x-y|^{3}(1+|x|)^{3}}\Big|v(y)\phi(y)\Big|{\rm d}y.
\end{split}
\end{equation*}
Thus by Lemma \ref{Reisz-potential-boundedness}, we know $\psi\in L^{2}(\mathbf{R}^{7})$.

Next, we assume $\psi\in L^{2}(\mathbf{R}^{7})$ and $\phi=Uv\psi$ and show
\begin{equation}\label{S2-7}
\int_{\mathbf{R}^{7}}v(y)\phi(y){\rm d}y=0.
\end{equation}
Since
\begin{equation*}
\begin{split}
\psi=\int_{\mathbf{R}^{7}}\bigg[\frac{1}{|x-y|^{3}}-\frac{1}{(1+|x|)^3}\bigg]v(y)\phi(y){\rm d}y+\int_{\mathbf{R}^{7}}\frac{1}{(1+|x|)^3}v(y)\phi(y){\rm d}y.
\end{split}
\end{equation*}
If $\psi\in L^{2}(\mathbf{R}^{7})$, then
\begin{equation*}
\frac{1}{(1+|x|)^{3}}\int_{\mathbf{R}^{7}}v(y)\phi(y){\rm d}y\in L^{2}(\mathbf{R}^{7}).
\end{equation*}
Hence \eqref{S2-7} holds and $\phi\in S_{2}L^{2}$.
\end{proof}

\begin{proof}[{\bf Proof of Theorem \ref{classification-8} ($d=8$).}]
	(1) We first note that
	\[[(-\Delta)^2+V]g=0 \Leftrightarrow (I+G_0V)g=0.\]
	Firstly, if $\phi\in S_1L^2\setminus\{0\}.$ Then $(U+vG_0v)f=0,$ and one has
	\[\phi(x)=-UvG_0v\phi=-c_1Uv(x)\int_{\mathbf{R}^8}\frac{v(y)f(y)}{|x-y|^4}{\rm d}y.\]
	Accordingly, we define
	\[\psi(x)=-G_0vf=c_1\int_{\mathbf{R}^8}\frac{v(y)f(y)}{|x-y|^4}{\rm d}y.\]
	Since $vf\in L^{2}_{3+}(\mathbf{R}^8),$ we have that $\psi\in W_{0}(\mathbf{R}^8)$ by  Lemma \ref{Reisz-potential-boundedness}. Further $\phi(x)=Uv(x)\psi(x)$ and
	\[\psi(x)=-G_0v\phi=-G_0V\psi(x) \Rightarrow (I+G_0V)\psi=0.\]
	Secondly, assume $\phi=Uv\psi$ for a non-zero distributional solution to $[(-\Delta)^2+V]\psi=0.$ It is clear that $\phi\in L^{2}_{4}(\mathbf{R}^8)$ and now
	\[(U+vG_0v)\phi(x)=v(x)\psi(x)+v(x)G_0V\psi(x)=v(x)(I+G_0V)\psi(x)=0.\]
	Hence, $\phi\in S_1L^2\setminus\{0\}.$

	(2) Assume first that $\phi\in S_2L^2\setminus\{0\}$. Since $S_2\le S_1$ and $P\phi=0$, then
	\[\psi(x)=-c_1\int_{\mathbf{R}^8}\Big[\frac{1}{|x-y|^4}-\frac{1}{1+|x|^4}\Big]v(y)f(y){\rm d}y.\]
	Using
	\[\Big|\frac{1}{|x-y|^4}-\frac{1}{1+|x|^4}\Big|
	\lesssim\frac{(1+|y|)^4}{(1+|x|)^4|x-y|^4}+\frac{(1+|y|)^2}{|x-y|^4(1+|x|)^2}+\frac{1+|y|}{|x-y|^4(1+|x|)},\]
	then Lemma \ref{Reisz-potential-boundedness} shows that $\psi\in L^{2}(\mathbf{R}^8)$ as desired.\\
	\indent	On the other hand, if $\phi=Uv\psi$ as in hypothesis, we have
	\begin{equation}\label{eq-s2-1-8}
	\psi(x)=-c_1\int_{\mathbf{R}^8}\Big[\frac{1}{|x-y|^4}-\frac{1}{1+|x|^4}\Big]v(y)f(y){\rm d}y-\frac{c_1}{1+|x|^4}\int_{\mathbf{R}^8}v(y)f(y){\rm d}y.
	\end{equation}
	The first term and $\psi(x)$ are in $W_{0}(\mathbf{R}^8),$ thus we must have that
	\[-\frac{c_1}{1+|x|^4}\int_{\mathbf{R}^8}v(y)f(y){\rm d}y\in W_{0}(\mathbf{R}^8),\]
	then
	\[\int_{\mathbf{R}^8}v(y)f(y){\rm d}y=0,\]
	that is $0=P\phi=S_1PS_1\phi$ and $\phi\in S_2L^2$ as desired.
\end{proof}	

Next, we show that for $d\geq9$, zero is only possible to be an eigenvalue of $H=(-\Delta)^2+V$.
\begin{proof}[\bf  Proof of Theorem \ref{classification-9} ($d\geq9$).]
	Since $\phi\in S_{1}L^{2}$, then
	\begin{equation*}
	0=(U+vG_{0}v)\phi=U\phi+vG_{0}v\phi,
	\end{equation*}
	so that $\phi=Uv(-G_{0}v\phi)$.
	
	Since $(-\Delta)^2G_{0}v\phi=v\phi$ distributionally, then
	\begin{equation*}
	[(-\Delta)^2+V](-G_{0}v\phi)=-v\phi+V(-G_{0}v\phi)=-vUv\psi+V\psi=0.
	\end{equation*}
	Now, we show $\psi\in L^{2}(\mathbf{R}^{d})$.  We need to show $-G_{0}v\phi\in L^{2}(\mathbf{R}^{d})$. Since
	\begin{equation*}
	\begin{split}
	G_{0}v\phi=&c\int_{\mathbf{R}^{d}}|x-y|^{4-d}v(y)\phi(y){\rm d}y=I_{4}v(y)\phi(y).
	\end{split}
	\end{equation*}
	Hence, by Lemma \ref{Reisz-potential-boundedness}, $-G_{0}v\phi\in L^{2}(\mathbf{R}^{d})$.
\end{proof}

{\bf Acknowledgements:} The authors were supported by  NSFC (11771165) and the program for Changjiang Scholars and Innovative Research Team in University (IRT13066). We would like to thank Professor  Avy  Soffer for his helps and useful discussions.

\end{document}